\documentclass[12pt]{amsart}
\makeatletter

\@addtoreset{equation}{section}
\makeatother
\usepackage[margin=3cm]{geometry}
\newenvironment{nouppercase}{%
  \renewcommand{\uppercasenonmath}[1]{}}{}
\usepackage{amsmath,amssymb,amsthm,color,shuffle}
\usepackage{mathrsfs} 
\usepackage[T1]{fontenc}
\usepackage{ytableau}
\ytableausetup{boxsize=normal,aligntableaux=center}  
\usepackage[all]{xy}
\allowdisplaybreaks[4]

\theoremstyle{definition}
\newtheorem{theorem}{Theorem}[section]
\newtheorem{prop}[theorem]{Proposition}
 
\newtheorem*{conj*}{Conjecture}
\newtheorem*{theorem*}{Theorem}
\newtheorem{remark}[theorem]{Remark} 
\newtheorem{cor}[theorem]{Corollary}

\newtheorem{lemma}[theorem]{Lemma}
\newtheorem{example}[theorem]{Example}

\usepackage{blkarray}

\definecolor{gray}{RGB}{230,230,230}  

\DeclareMathOperator{\add}{add}
\DeclareMathOperator{\rem}{rem}
\DeclareMathOperator{\diag}{diag}

\DeclareMathOperator{\T}{T}
\DeclareMathOperator{\SSYT}{SSYT}

\DeclareMathOperator{\IB}{IB}
\DeclareMathOperator{\OB}{OB}
 
\newcommand{\ba}{\boldsymbol a}
\newcommand{\bk}{\boldsymbol k}
\newcommand{\bs}{\boldsymbol s}

\newcommand{\bu}{\boldsymbol u}
\newcommand{\bv}{\boldsymbol v}

\newcommand{\ol}[1]{\overline{#1}}
\newcommand{\mb}[1]{\raisebox{-2pt}{$\ol{#1}$}}

\newcommand{\FK}{\mathrm{FK}}

\newcommand{\tcdots}{\text{\tiny \hspace{1.5pt}$\cdots$}}
\newcommand{\tvdots}{\text{\tiny \raisebox{-1.5pt}{$\vdots$}}}
\newcommand{\tddots}{\text{\tiny \raisebox{-1.5pt}{$\ddots$}}}

\usepackage{tikz}
\usetikzlibrary{calc}
\usetikzlibrary{decorations.pathreplacing}

\tikzset{vertex_black/.style={
circle, draw, black, fill=black,
inner sep=0pt, minimum width=4pt, label distance=1mm }}
\tikzset{vertex_white/.style={
circle, draw, black, fill=white,
inner sep=0pt, minimum width=4pt, label distance=1mm }}
\tikzset{vertex_dots/.style={yshift=0,inner sep=2,minimum width=0pt}}

\begin{document}

\title{
Quadratic relations for ninth variations of Schur functions and \\ application to Schur multiple zeta functions
}
\author[W. Takeda]{Wataru Takeda}
\address{Wataru Takeda\\ Department of Mathematics, Toho University,
2-2-1, Miyama, Funabashi-shi, Chiba 274-8510, Japan.}
\email{wataru.takeda@sci.toho-u.ac.jp}
\author[Y. Yamasaki]{Yoshinori Yamasaki}
\address{Graduate School of Science and Engineering, Ehime University, 2-5, Bunkyo-cho, Matsuyama 790-8577, Japan}
\email{yamasaki@math.sci.ehime-u.ac.jp}

\begin{nouppercase}
\maketitle
\end{nouppercase}
\begin{abstract}
Macdonald's ninth variation of Schur functions is a broad generalization of the classical Schur function and its variants, defined via the Jacobi–Trudi determinant formula.
In this paper, we establish various algebraic relations for $S^{(r)}_{\lambda/\mu}(X)$, a class of the ninth variation introduced by Nakagawa, Noumi, Shirakawa, and Yamada, by combining the Jacobi–Trudi formula with determinant formulas such as the Desnanot–Jacobi adjoint matrix theorem and the Pl\"ucker relations, which generalize the corresponding relations for Schur functions.
As an application, we investigate algebraic relations for "diagonally constant" Schur multiple zeta functions and examine their specific special values when the shape is rectangular.
\end{abstract}


\section{Introduction}

Schur functions $s_{\lambda}$, along with their skew generalization $s_{\lambda/\mu}$ called skew Schur functions, 
are an important class of symmetric functions.
In 1992, Macdonald \cite{ma92} introduced the so-called ninth variation of Schur functions, 
which includes many variants of $s_{\lambda/\mu}$ such as the factorial and flagged Schur functions,
defined as those satisfying the Jacobi-Trudi formula. 
About a decade later, Nakagawa, Noumi, Shirakawa, and Yamada \cite{nnsy}
studied a class of the ninth variation $S^{(r)}_{\lambda/\mu}(X)$ of Schur functions,
defined for a matrix of variables $X$ via its Gauss decomposition.
They established several determinant formulas and tableau expressions (in some special cases of $X$) for $S^{(r)}_{\lambda/\mu}(X)$,
using properties of minor determinants.
Recently, Foley and King \cite{fk21} introduced another type of the ninth variation of Schur functions, denoted by $S^{\FK}_{\lambda/\mu}(W)$, defined via a tableau expression, and derived the Hamel–Goulden formula \cite{hg}, which gives an extensive generalization of the Jacobi–Trudi formula.
We remark that, as will be seen in Lemma~\ref{lemma:nnsy to fk}, $S^{\FK}_{\lambda/\mu}(W)$ can be obtained as a specialization of $S^{(r)}_{\lambda/\mu}(X)$. 

It is known that Schur functions satisfy various algebraic relations.
For example, by combining the Jacobi–Trudi formula with the Pl\"ucker relations for products of determinants, Kleber \cite{k01} derived a quadratic relation for $s_{\lambda}$, described combinatorially in terms of the Young diagram (see Theorem~\ref{thm:kPlucker} for details).
The aim of the present paper is to establish algebraic (mainly quadratic) relations for $S^{(r)}_{\lambda/\mu}(X)$, including a generalization of Kleber’s formula.
As a direct consequence, we obtain the corresponding algebraic relations for $S^{\FK}_{\lambda/\mu}(W)$ via the specialization mentioned above.
Moreover, we apply these results to the Schur multiple zeta functions $\zeta_{\lambda/\mu}(\boldsymbol{s})$ introduced in \cite{npy}, which provide a combinatorial generalization of both multiple zeta and multiple zeta-star functions, particularly when the variable tableau $\boldsymbol{s}$ is diagonally constant.
Actually, these further allow us to derive explicit expressions or generating functions for certain families of Schur multiple zeta values of rectangular shape.



The organization of the paper is as follows.  
In Section~\ref{sec:nnsy9th}, we first review the definition of $S^{(r)}_{\lambda/\mu}(X)$,  
and then prove that it satisfies the Giambelli formula (Theorem~\ref{thm:Giambelli nnsy}),  
which gives a skew generalization of the formula for $S^{(r)}_{\lambda}(X)$ obtained in \cite{nnsy}. Next, we recall the definition of $S^{\FK}_{\lambda/\mu}(W)$ 
and show that it can be obtained as a specialization of $S^{(r)}_{\lambda/\mu}(X)$.
Sections~\ref{djformulasection} and~\ref{pformula} are devoted to investigating several algebraic relations for $S^{(r)}_{\lambda/\mu}(X)$, derived from the Desnanot–Jacobi adjoint matrix theorem and the Pl\"ucker relations, respectively.
In particular, by introducing the adding and removing operators for Young diagrams, we generalize Kleber’s quadratic relation for $s_{\lambda}$ \cite{k01} to $S^{(r)}_{\lambda}(X)$ (Theorem~\ref{thm:9thPlucker}).
Finally, in Section \ref{application}, as an application or related topic of the results obtained in the previous sections, we study relations for the "diagonally constant" Schur multiple zeta functions $\zeta_{\lambda/\mu}(\boldsymbol{s})$. In particular, we focus on the case where $\lambda/\mu$ is a rectangular shape and the variable tableau $\boldsymbol{s}$ is filled with at most three integers, noting that in non-admissible cases, we consider the regularization studied in \cite{bc}.
For example, we express the generating function for such values via the generating function of the corresponding (regularized) multiple zeta values.


\section{Ninth variations of skew Schur functions}
\label{sec:nnsy9th}

In this section, we recall the definitions of the ninth variations of skew Schur functions introduced in \cite{nnsy} and \cite{fk21}, respectively.
Moreover, for the former, we derive a Giambelli formula that extends the result of \cite{nnsy} from the non-skew case.

\subsection{Notations and Terminologies}

We first summarize the notations and terminologies that are used throughout the present paper.
A partition is a non-increasing sequence $\lambda=(\lambda_1,\ldots,\lambda_k)$ of positive integers.
The length and weight of $\lambda$ are denoted by  $\ell(\lambda):=k$ and $|\lambda|:=\lambda_1+\cdots+\lambda_k$, respectively. 
We sometimes write $\lambda=(\lambda_1,\lambda_2,\ldots)$ with the understanding that $\lambda_i=0$ for $i>k$, and also $\lambda=(1^{m_1(\lambda)}2^{m_2(\lambda)}\cdots)$,
where $m_i(\lambda)$ is the multiplicity of $i$ in $\lambda$.
The Young diagram associated with $\lambda$ is defined by 
$D({\lambda}):=\{(i,j)\in {\mathbb{Z}}^2\,|\,1\leq i\leq k, 1\leq j\leq \lambda_i\}$, which is depicted as a collection of square boxes with the $i$-th row having $\lambda_i$ boxes.
The conjugate partition of $\lambda$ is denoted by 
$\lambda'=(\lambda'_1,\lambda'_2,\ldots,\lambda'_{k'})$,
where $\lambda'_i=\#\{j\,|\,\lambda_j\ge i\}$.
A skew partition $\lambda/\mu$ is a pair of partitions 
$\lambda=(\lambda_1,\ldots,\lambda_k)$ and 
$\mu=(\mu_1,\ldots,\mu_l)$ satisfying $\lambda\supset\mu$, that is, $k\ge l$ and $\lambda_i\geq \mu_i$ for all $i$. 
When $\mu$ is the empty partition $\varnothing$, 
we identify $\lambda/\mu$ with $\lambda$. 
We also associate $\lambda/\mu$ with the skew Young diagram $D({\lambda/\mu}):=D(\lambda)\setminus D(\mu)$.
We say that $(i,j)\in D({\lambda/\mu})$ is a corner of $\lambda/\mu$ if $(i+1, j)\notin D({\lambda/\mu})$ and $(i, j+1)\notin D({\lambda/\mu})$, 
and denote by $C({\lambda/\mu})$ the set of all corners of $\lambda/\mu$.
A Young tableau of shape $\lambda/\mu$ over a set $S$ is a filling $T=(t_{i,j})_{(i,j)\in D(\lambda/\mu)}$ of boxes of $D(\lambda/\mu)$ with $t_{i,j}\in S$.
We denote by $\T(\lambda/\mu,S)$ the set of all  Young tableaux of shape $\lambda/\mu$ over $S$. 
Finally, for a positive integer $n$, we put $[n]:=\{1,2,\ldots,n\}$.

\subsection{Ninth variations of skew Schur functions defined in \cite{nnsy}}

From now on, we always assume that a partition $\lambda$ is contained in the rectangle $(s^r)$ for some non-negative integers
$r$ and $s$. Put $N=r+s$.

Let $X=[x_{i,j}]_{1\le i,j\le N}$ be an $N$-by-$N$ matrix, where each $(i,j)$-entry $(X)_{i,j}=x_{i,j}$ of $X$ is assumed to be indeterminate.
For sequences of row indices $A=(a_1,\ldots,a_r)$ and column indices $B=(b_1,\ldots,b_r)$, 
let $X^{A}_{B}=X^{a_1,\ldots,a_r}_{b_1,\ldots,b_r}:=[x_{a_i,b_j}]_{1\le i,j\le r}$
be the $r$-by-$r$ submatrix of $X$ 
corresponding to $A$ and $B$
and $\xi^{A}_{B}(X)=\xi^{a_1,\ldots,a_r}_{b_1,\ldots,b_r}(X):=\det X^{A}_{B}$.
Moreover, for subsets $I,J\subset [N]$ with $|I|=|J|=r$, 
we put $X^{I}_J:=X^{i_1,\ldots,i_r}_{j_1,\ldots,j_r}$
and $\xi^{I}_{J}(X):=\xi^{i_1,\ldots,i_r}_{j_1,\ldots,j_r}(X)$
where $i_1<\cdots <i_r$ and $j_1<\cdots<j_r$ are the sequences obtained by arranging the elements of $I$ and $J$ in increasing order, respectively.
To calculate minors, the following formulas are useful:
\begin{align}
\label{for:xi formula}
 \xi^{I}_{J}(XY)
&=\sum_{\substack{K\subset [N] \\ |K|=r}}\xi^{I}_{K}(X)\xi^{K}_{J}(Y),\\
\label{for:Jacobi complement}
 \xi^{I}_{J}(X)
&=(-1)^{\sum_{i\in I} i+\sum_{j\in J} j}\cdot \det X\cdot \xi^{J^c}_{I^c}(X^{-1}),
\end{align}
where $I^c=[N]\setminus I$ and $J^c=[N]\setminus J$, respectively.
The second one is called Jacobi's complementary minor formula.
Write the Gauss decomposition of $X$ as 
\[
 X=X_{-}X_{0}X_{+},
\]
where $X_{-}$, $X_{0}$ and $X_{+}$ are lower unitriangular,
diagonal and upper unitriangular matrices, respectively,
which are uniquely determined as matrices with entries in 
$\mathbb{C}(X)$, the field of rational functions over $\mathbb{C}$ in the variable $x_{i,j}$ for $1\le i,j\le N$. 
Actually, one sees that 
\begin{align}
\label{for:X-}
 (X_{-})_{i,j} 
&=\frac{\xi^{1,\ldots,j-1,i}_{1,\ldots,j-1,j}(X)}{\xi^{1,\ldots,j}_{1,\ldots,j}(X)} & & \hspace{-50pt} (i\ge j),\\
\label{for:X0}
(X_{0})_{i,i} 
&=\frac{\xi^{1,\ldots,i}_{1,\ldots,i}(X)}{\xi^{1,\ldots,i-1}_{1,\ldots,i-1}(X)} 
& & \hspace{-50pt} (1\le i\le N),\\
\label{for:X+}
(X_{+})_{i,j} 
&=\frac{\xi^{1,\ldots,i-1,i}_{1,\ldots,i-1,j}(X)}{\xi^{1,\ldots,i}_{1,\ldots,i}(X)} & & \hspace{-50pt} (i\le j).
\end{align}
Notice that the entries of the inverse matrix $X^{-1}_{+}$ of $X_{+}$ is similarly given by 
\begin{align}
\label{for:X+inverse}
 (X^{-1}_{+})_{i,j} 
&=(-1)^{j-i}\frac{\xi^{1,\ldots,j-1}_{1,\ldots,\widehat{i},\ldots,j}(X)}{\xi^{1,\ldots,j-1}_{1,\ldots,j-1}(X)} \qquad  (i\le j),
\end{align}
where $\widehat{i}$ means that we ignore $i$.
In \cite{nnsy}, the ninth variation of skew Schur function
$S^{(r)}_{\lambda/\mu}(X)$ is defined by 
\begin{align}
\label{def:9thSchur-nnsy}
 S^{(r)}_{\lambda/\mu}(X)
:=\xi^{I}_{J}(X_{+})
=(-1)^{|\lambda/\mu|}\xi^{J^c}_{I^c}(X^{-1}_{+})\in \mathbb{C}(X),
\end{align}
where $I=\{i_1,\ldots,i_r\}\subset [N]$ 
with $i_a=\mu_{r+1-a}+a$ 
and $J=\{j_1,\ldots,j_r\}\subset [N]$ 
with $j_a=\lambda_{r+1-a}+a$ 
are the Maya diagrams of $\mu$ and $\lambda$, respectively.
Notice that the second equality in \eqref{def:9thSchur-nnsy} follows from \eqref{for:Jacobi complement},
and $I^{c}=\{k_1,\ldots,k_s\}\subset [N]$ with $k_a=r+a-\mu'_a$
and $J^c=\{l_1,\ldots,l_s\}\subset [N]$ with 
$l_a=r+a-\lambda'_a$.
As the special case $\mu=\varnothing$,
using \eqref{for:xi formula} 
with \eqref{for:X-}, \eqref{for:X0} and \eqref{for:X+}, we have   
\begin{align}
\label{for:Weyl formula}
 S^{(r)}_{\lambda}(X)
=\frac{\xi^{1,\ldots,r}_{j_1,\ldots,j_r}(X)}{\xi^{1,\ldots,r}_{1,\ldots,r}(X)}.
\end{align}
This is a kind of the classical Weyl formula for Schur function.
We remark that 
$S^{(r)}_{\lambda/\mu}(X)$ gives
the classical skew Schur function $s_{\lambda/\mu}(x_1,\ldots,x_n)$ of $n$ variables 
when $X$ is the Vandermonde matrix 
$X=\left[x^{j-1}_{i}\right]_{1\le i,j\le N}$ 
with variables $\{x_i\}_{i\in [N]}$,
under the specialization $x_{n+1}=\cdots=x_{N}=0$ when $r$ and $s$ are sufficiently large.
In the following, for simplicity, we graphically express $S^{(r+m)}_{\lambda/\mu}(X)$ for $m\in\mathbb{Z}$ 
by using the Young tableau 
$
\left(\,\overline{m+c(i,j)}\,\right)_{(i,j)\in D(\lambda/\mu)}
$
of shape $\lambda/\mu$.
Here, $c(i,j):=j-i$ is the content of $(i,j)\in D(\lambda/\mu)$, and, for $n\in\mathbb{Z}$, $\overline{n}=n$ if $n\ge 0$ and $-|n|$ otherwise.
For example, 
\[
S^{(r+1)}_{(3,3,2,1)/(1,1)}(X)
=
\ 
\ytableausetup{boxsize=12pt,aligntableaux=center}
\scriptsize
\begin{ytableau}
\none &2&3\\
\none &1&2\\
\mb{1}&0\\
\mb{2}
\end{ytableau},
\qquad 
S^{(r-2)}_{(4,3,3,2)/(2,1,1)}(X)
=
\ 
\begin{ytableau}
\none & \none &0&1\\
\none &\mb{2}&\mb{1}\\
\none &\mb{3}&\mb{2} \\
\mb{5}&\mb{4}
\end{ytableau}.
\]

As special cases, we put 
\begin{align}
\label{def:h}
 h^{(r)}_{d}(X)
&:=S^{(r)}_{(d)}(X)
=\xi^{1,\ldots,r}_{1,\ldots,r-1,r+d}(X_{+}) \qquad (0\le d\le s),\\
\label{def:e}
e^{(r)}_{d}(X)
&:=S^{(r)}_{(1^d)}(X)
=\xi^{1,\ldots,r}_{1,\ldots,\widehat{r-d+1},\ldots,r+1}(X_{+}) \qquad (0\le d\le r),
\end{align}
and $h^{(r)}_{d}(X)=e^{(r)}_{d}(X)=0$ if $d<0$.
Notice from \eqref{for:X+}, \eqref{for:X+inverse} and 
\eqref{for:Weyl formula} that 
\begin{align}
\label{for:Xij_X-ij}
 (X_{+})_{i,j} 
=h^{(i)}_{j-i}(X),\qquad 
 (X^{-1}_{+})_{i,j} 
=(-1)^{j-i}e^{(j-1)}_{j-i}(X).    
\end{align}
These observations enable us to obtain the following Jacobi-Trudi formulas
for $S^{(r)}_{\lambda/\mu}(X)$.

\begin{theorem}
[{\cite[Theorem~1.1 and (1.25)]{nnsy}}]
\label{thm:JT-nnsy}
\ 
\begin{enumerate}
\item (Jacobi-Trudi formula)
We have 
\begin{equation}
\label{for:9thJT-nnsy} 
S^{(r)}_{\lambda/\mu}(X)
=\det H^{(r)}_{\lambda/\mu}(X),
\quad   
H^{(r)}_{\lambda/\mu}(X)
:=\left[h^{(r+\mu_j-j+1)}_{\lambda_i-\mu_j-i+j}(X)\right]_{1\le i,j\le \ell(\lambda)}.
\end{equation}
\item (Dual Jacobi-Trudi formula)
We have 
\begin{equation}
\label{for:9thDJT-nnsy} 
S^{(r)}_{\lambda/\mu}(X)
=\det E^{(r)}_{\lambda/\mu}(X),
\quad 
E^{(r)}_{\lambda/\mu}(X)
:=\left[e^{(r-\mu'_j+j-1)}_{\lambda_i'-\mu_j'-i+j}(X)\right]_{1\le i,j\le \ell(\lambda')}.
\end{equation}
\end{enumerate}
\end{theorem}

\begin{example}

When $\lambda=(2,2,1)$ and $\mu=(1)$, we have 
\[
\ytableausetup{boxsize=12pt,aligntableaux=center}
\scriptsize
\begin{ytableau}
\none & 1 \\
\mb{1} & 0 \\
\mb{2}  
\end{ytableau}
\,
=
\,
\left|~
\begin{array}{ccc}
\begin{ytableau}
1
\end{ytableau}
 &
\begin{ytableau}
\mb{1}&0&1
\end{ytableau}
 &
\begin{ytableau}
\mb{2}&\mb{1}&0&1
\end{ytableau}
\\[12mm]
1
&
\begin{ytableau}
\mb{1}&0
\end{ytableau}
&
\begin{ytableau}
\mb{2}&\mb{1}&0
\end{ytableau}
\\[12mm]
0
&
1
&
\begin{ytableau}
\mb{2}
\end{ytableau}
\end{array}
~\right|
=
\,
\left|~
\begin{array}{cc}
\begin{ytableau}
\mb{1}\\
\mb{2}
\end{ytableau}
 &
\begin{ytableau}
1\\
0\\
\mb{1}\\
\mb{2} 
\end{ytableau}
\\[12mm]
1
 &
\begin{ytableau}
1\\
0
\end{ytableau}
\end{array}
~\right|.
\]
\end{example}

\subsection{Giambelli formula for $S^{(r)}_{\lambda/\mu}(X)$}
\label{sebsec:Giambelli}
In this section, we prove the Giambelli formula for $S^{(r)}_{\lambda/\mu}(X)$.
To do that, we first review the Frobenius notation.
Let $\lambda$ be a partition having $p$ diagonal entries. Define the sequences of indices $\alpha_1>\cdots>\alpha_p\ge 0$ and $\beta_1>\cdots>\beta_p\ge 0$ by $\alpha_i=\lambda_i-i$ and $\beta_i=\lambda'_i-i$ for $1\le i\le p$.
Then, in the Frobenius notation, we write $\lambda=(\alpha\,|\,\beta)$ with $\alpha=(\alpha_1,\ldots,\alpha_p)$ and $\beta=(\beta_1,\ldots,\beta_p)$.
Let $\mu$ be a partition satisfying $\lambda\supset\mu$ and 
$\mu=(\gamma\,|\,\delta)$ the Frobenius notation of $\mu$ with $\gamma=(\gamma_1,\ldots,\gamma_q)$ and $\delta=(\delta_1,\ldots,\delta_q)$. 
Notice that $p\ge q$ and both $\alpha_i\ge \gamma_i$ and $\beta_i\ge\delta_i$ hold for $1\le i\le q$.

\begin{theorem}
\label{thm:Giambelli nnsy}
Retaining the notations above, we have
\begin{equation}
\label{for:Gambelli nnsy}
 S^{(r)}_{\lambda/\mu}(X)
=(-1)^q\det G^{(r)}_{\lambda/\mu}(X),
\quad 
G^{(r)}_{\lambda/\mu}(X)
:=\left[
\begin{array}{c|c}
\left[S^{(r)}_{(\alpha_i\,|\,\beta_j)}(X)\right]_{\substack{1\le i\le p \\ 1\le j\le p}} & 
\left[h^{(r+\gamma_j+1)}_{\alpha_i-\gamma_j}(X)\right]_{\substack{1\le i\le p \\ 1\le j\le q}} \\[10pt]
\hline \\[-10pt]
 \left[e^{(r-\delta_i-1)}_{\beta_j-\delta_i}(X)\right]_{\substack{1\le i\le q \\ 1\le j\le p}} & O_{q}
\end{array}
\right]
\end{equation}
with $O_{q}$ being the square zero matrix of size $q$.
\end{theorem}

We notice that this was proved in \cite[Theorem~1.2]{nnsy} 
only when $\mu=\varnothing$.
To prove this in general situations, 
we use the following formula concerning minor determinants.

\begin{theorem}
[{\cite{b}, cf. \cite[Corollary~3.2]{o21}}]
\label{thm:Bazin}
Let $m,n$ be positive integers such that $m\le n$.
Let $Z$ be a matrix having $n$ rows,
and $A=(a_1,\ldots,a_m)$, $B=(b_1,\ldots,b_m)$ and $C$ 
be sequences of column indices of length $m$, $m$ and $n-m$, respectively. 
Then, it holds that  
\begin{align}
\label{for:Bazin}
 \det\left[\xi^{(1,\ldots,n)}_{(a_{i})\,\sqcup\,(B\setminus(b_j))\,\sqcup\, C}(Z)\right]_{1\le i,j\le m}
 =(-1)^{\frac{m(m-1)}{2}}\xi^{(1,\ldots,n)}_{A\,\sqcup\,C}(Z)\left(\xi^{(1,\ldots,n)}_{B\,\sqcup\,C}(Z)\right)^{m-1}.
\end{align}
Here, $B\setminus(b_j)=(b_1,\ldots,\widehat{b_{j}},\ldots,b_m)$
and $\sqcup$ means the concatenation. 
\end{theorem}

\begin{proof}
[Proof of Theorem~\ref{thm:Giambelli nnsy}]
The proof of \eqref{for:Gambelli nnsy} given here is essentially the same as that of \cite{lp84}
for the classical skew Schur functions. 
We first notice that the Maya diagrams $I$ and $J$ of $\mu$ and $\lambda$ in terms of their Frobenius notations are respectively given by 
\begin{align*}
 I&=\{1,\ldots,r\}\cup\{r+\gamma_q+1,\ldots,r+\gamma_1+1\}
\setminus\{r-\delta_1,\ldots,r-\delta_q\},\\
 J&=\{1,\ldots,r\}\cup\{r+\alpha_p+1,\ldots,r+\alpha_1+1\}
\setminus\{r-\beta_1,\ldots,r-\beta_p\}.
\end{align*}
From \eqref{def:9thSchur-nnsy} and \eqref{for:Xij_X-ij}, 
we have 
\begin{align}
\label{for:9thJT-nnsy-h2}
 S^{(r)}_{\lambda/\mu}(X)
 =\xi^{I}_{J}(H),
\end{align}
where $H:=[h^{(i)}_{j-i}(X)]_{i,j\ge 1}$.
Define the $(r+q)$-by-$(r+p)$ submatrix $H'$ of $H$ by 
\[
H'=H^{(1,\ldots,r)\,\sqcup\,(r+\gamma_q+1,\ldots,r+\gamma_1+1)}_{(1,\ldots,r)\,\sqcup\,(r+\alpha_p+1,\ldots,r+\alpha_1+1)},
\]
and the $(r+q)$-by-$2q$ matrix 
$K=[k_{i,j}]_{\substack{1\le i\le r+q \\ 1\le j\le 2q}}$
by 
\[
 k_{i,j}=
 \begin{cases}
 \delta_{i,r-\delta_{q+1-j}}& (1\le i\le r,\ 1\le j\le q),\\
 \delta_{q+i,r+j} & (r+1\le i\le r+q,\ q+1\le j\le 2q), \\
 0 & \text{otherwise}
 \end{cases}
\]
with $\delta_{i,j}$ being the Kronecker delta.
Now we apply Theorem~\ref{thm:Bazin} to the $(r+q)\times (r+p+2q)$ matrix $Z=[\,H'\,|\,K\,]$ with
\begin{align*}
A&=(a_1,\ldots,a_{p+q}), \quad 
a_{i}
=r+i \quad (1\le i\le p+q),\\
B
&=(b_1,\ldots,b_{p+q}), \quad 
b_{i}
=\begin{cases}
 r-\beta_{p+1-i} & (1\le i\le p),\\
 r+i+q & (p+1\le i\le p+q),
\end{cases}
\\
C&=(1,\ldots,r)\setminus (r-\beta_1,\ldots,r-\beta_p).
\end{align*}
Reordering $A\,\sqcup\,C$ and $B\,\sqcup\,C$ in increasing order and applying cofactor expansion yield
\begin{align*}
 \xi^{(1,\ldots,n)}_{A\,\sqcup\,C}(Z)
=(-1)^{\varepsilon_{A,C}}S^{(r)}_{\lambda/\mu}(X),
\quad 
\xi^{(1,\ldots,n)}_{B\,\sqcup\,C}(Z)
=(-1)^{\varepsilon_{B,C}},
\end{align*}
where 
$\varepsilon_{A,C}:=(p+q)(r-p)+\sum^q_{k=1}(r-\delta_{q+1-k}+r+1)$
and 
$\varepsilon_{B,C}:=\sum^{p}_{k=1}(q+r-\beta_{p+1-k}-1)+rq$.
This shows that the right-hand side of \eqref{for:Bazin} is 
$(-1)^{\varepsilon}S^{(r)}_{\lambda/\mu}(X)$
with 
$\varepsilon
:={\frac{(p+q)(p+q-1)}{2}}+\varepsilon_{A,C} +(p+q-1)\varepsilon_{B,C}$.

We next compute the left-hand side of \eqref{for:Bazin}.
Put 
$d_{i,j}=\xi^{(1,\ldots,n)}_{(a_{i})\,\sqcup\,(B\setminus (b_j))\,\sqcup\, C}(Z)$.
\begin{itemize}
\item 
When $1\le i,j\le p$, we have from \eqref{for:9thJT-nnsy-h2} $d_{i,j}
=(-1)^{u_j} S^{(r)}_{(\alpha_{p+1-i}\,|\,\beta_{p+1-j})}(X),
$
where $u_{j}:=(r-p)q+r-1
+\sum^{j-1}_{k=1}(r-\beta_{p+1-k}-2)
+\sum^{p}_{k=j+1}(r-\beta_{p+1-k}-1)$.
\item 
When $1\le i\le q$ and $1\le j\le p$,
we have  
$d_{p+i,j}
=(-1)^{u_j+\delta_{q+1-i}}e^{(r-\gamma_{q+1-i}-1)}_{\beta_{p+1-j}-\delta_{q+1-i}}(X)$.
\item 
When $1\le i\le p$ and $1\le j\le q$, we have 
${d_{i,p+j}}
=(-1)^{v_j+j+1}h^{(r+\gamma_{q+1-j}+1)}_{\alpha_{p+1-i}-\gamma_{q+1-j}}(X)$,
where $v_j:=(r-p)(q-1)+r+\sum^{p}_{k=1}(r-\beta_{p+1-k}-1)$.
\item 
When $1\le i,j\le q$, we have $d_{p+i,q+j}=0$.
\end{itemize}
Therefore, the left-hand side of \eqref{for:Bazin} is  
\begin{align*}
\det\left[
\begin{array}{c|c}
\left[(-1)^{u_j}S^{(r)}_{(\alpha_{p+1-i}\,|\,\beta_{q+1-j})}(X)\right]_{\substack{1\le i\le p \\ 1\le j\le p}} & 
\left[(-1)^{v_j}h^{(r+\gamma_{q+1-j}+1)}_{\alpha_{p+1-i}-\gamma_{q+1-j}}(X)\right]_{\substack{1\le i\le p \\ 1\le j\le q}} \\[10pt]
\hline \\[-10pt]
 \left[(-1)^{u_j+\delta_{q+1-i}}e^{(r-\delta_{q+1-i}-1)}_{\beta_{p+1-j}-\delta_{q+1-i}}(X)\right]_{\substack{1\le i\le q \\ 1\le j\le p}} & O_{q}
\end{array}
\right],
\end{align*}
which is equal to $(-1)^{\varepsilon'}
\det G^{(r)}_{\lambda/\mu}(X)$ with 
$\varepsilon':=\sum^{p}_{k=1}u_k +\sum^{q}_{k=1}v_k+\sum^{q}_{k=1}\delta_{q+1-k}+\sum^{q}_{k=1}(k+1)$.
Now the desired formula follows because
$\varepsilon-\varepsilon'\equiv q$ \!\!\! $\pmod{2}$.
\end{proof}

\begin{example}
When $\lambda=(5,5,5,3)=(4,3,2\,|\,3,2,1)$
and $\mu=(3,2)=(2,0\,|\,1,0)$, we have 
\[
\ytableausetup{boxsize=12pt,aligntableaux=center}
\scriptsize
\begin{ytableau}
\none & \none & \none & 3 & 4\\
\none & \none & 1 & 2 & 3 \\
\mb{2} & \mb{1} & 0 & 1 & 2\\
\mb{3} & \mb{2} & \mb{1} 
\end{ytableau}
\,
=
\,
\left|~
\begin{array}{ccc|cc}
\begin{ytableau}
0&1&2&3&4\\
\mb{1}\\
\mb{2} \\
\mb{3} 
\end{ytableau}
 &
\begin{ytableau}
0&1&2&3&4\\
\mb{1}\\
\mb{2} 
\end{ytableau}
 &
\begin{ytableau}
0&1&2&3&4\\
\mb{1}
\end{ytableau}
&
\begin{ytableau}
3&4
\end{ytableau}
&
\begin{ytableau}
1&2&3&4
\end{ytableau}
\\[12mm]
\begin{ytableau}
0&1&2&3\\
\mb{1}\\
\mb{2} \\
\mb{3} 
\end{ytableau}
 &
\begin{ytableau}
0&1&2&3\\
\mb{1}\\
\mb{2}
\end{ytableau}
 &
\begin{ytableau}
0&1&2&3\\
\mb{1}
\end{ytableau}
 &
\begin{ytableau}
3
\end{ytableau}
 &
\begin{ytableau}
1&2&3
\end{ytableau}
\\[12mm]
\begin{ytableau}
0&1&2\\
\mb{1}\\
\mb{2} \\
\mb{3}  
\end{ytableau}
 &
\begin{ytableau}
0&1&2\\
\mb{1}\\
\mb{2}
\end{ytableau} 
 &
\begin{ytableau}
0&1&2\\
\mb{1}
\end{ytableau}
&1&
\begin{ytableau}
1&2
\end{ytableau}\\[11mm]
\hline 
\raisebox{27pt}{}
\begin{ytableau}
\mb{2}\\
\mb{3}
\end{ytableau}
&
\begin{ytableau}
\mb{2}
\end{ytableau}
&1&0&0\\[10mm]
\begin{ytableau}
\mb{1}\\
\mb{2}\\
\mb{3}
\end{ytableau}
&
\begin{ytableau}
\mb{1}\\
\mb{2}
\end{ytableau}
&
\begin{ytableau}
\mb{1}
\end{ytableau}
&0&0
\end{array}
~\right|.
\]
\end{example}


\subsection{Ninth variations of skew Schur functions defined in \cite{fk21}}

We next present another type of the ninth variation of the skew Schur function defined by Foley and King.
In what follows, for a set $Y$,
we denote by $\mathbb{C}[Y]$ the ring of polynomials over $\mathbb{C}$ in indeterminates indexed by the elements of $Y$.


A Young tableau $(t_{i,j})\in \T(\lambda/\mu,\mathbb{N})$
with $\mathbb{N}$ being the set of positive integers is called semi-standard if it satisfies 
$t_{i,j}\le t_{i,j+1}$ and $t_{i,j}<t_{i+1,j}$
for all $i,j$.
We denote by $\SSYT(\lambda/\mu)$ the set of all semi-standard Young tableaux of shape $\lambda/\mu$. 
Moreover, for $M\in\mathbb{N}$,
let $\SSYT_M(\lambda/\mu)$ be the subset of $\SSYT(\lambda/\mu)$
consisting of all $(t_{i,j})$ satisfying $t_{i,j}\in [M]$ for all $i,j$.
The ninth variation of skew Schur function
$S^{\FK}_{\lambda/\mu}(W)$ 
with variables $W=\{w_{k,c}\}_{k\in [M],\,c\in \mathbb{Z}}$ introduced in \cite{fk21}
is defined by the following sum over all semi-standard tableaux:
\begin{align}
\label{def:9thFK}
S^{\FK}_{\lambda/\mu}(W)
:=\sum_{(t_{ij})\in \SSYT_M(\lambda/\mu)}\prod_{(i,j)\in D(\lambda/\mu)}w_{t_{i,j},c(i,j)}
\in\mathbb{C}[W].
\end{align}
This also gives a generalization of the  classical Schur function:
If $w_{k,c}$ does not depend on $c$,
then $S^{\FK}_{\lambda/\mu}(W)=s_{\lambda/\mu}(w_1,\ldots,w_M)$ where we write $w_k=w_{k,c}$.
As special cases, 
we put $h^{\FK}_d(W):=S^{\FK}_{(d)}(W)$ and $e^{\FK}_d(W):=S^{\FK}_{(1^d)}(W)$ for $d\in\mathbb{Z}_{>0}$, 
and $h^{\FK}_0(W)=e^{\FK}_0(W)=1$ and $h^{\FK}_d(W)=e^{\FK}_d(W)=0$ if $d<0$. 

Here, recall that 
$S^{(r)}_{\lambda/\mu}(X)$ 
also has a tableau expression when
$X$ is a special type of matrix.


\begin{theorem}
\label{thm:nnsy special}
({\cite[(2.59) and (2.64)]{nnsy}})
For $\bu=\{u^{(t)}_k\}_{k\in [M],\,t\in [N-1]}$ and $\bv=\{v^{(t)}_k\}_{k\in [M],\,t\in [N-1]}$, 
define the upper unitriangular matrices  
$U_M(\bu)$ and $V_M(\bv)$ of size $N$ by 
\begin{align*}
U_M(\bu):=U_1U_2\cdots U_M,
\qquad
V_M(\bv):=V_1V_2\cdots V_M,   
\end{align*}
where, 
\begin{align*}
 U_k
&=(E+u^{(1)}_{k}E_{1,2})(E+u^{(2)}_{k}E_{2,3})\cdots (E+u^{(N-1)}_{k}E_{N-1,N}),\\
 V_k
&=(E+v^{(N-1)}_{k}E_{N-1,N})\cdots (E+u^{(2)}_{k}E_{2,3})(E+u^{(1)}_{k}E_{1,2})
\end{align*}
with $E$ and $E_{i,j}$ being the unit matrix and the matrix unit of size $N$, respectively.
Then, we have 
\begin{align}
\label{for:tableau_U}
 S^{(r)}_{\lambda/\mu}(U_M(\bu))
&=\sum_{(t_{i,j})\in \SSYT_M(\lambda/\mu)}\prod_{(i,j)\in D(\lambda/\mu)}u^{(r+c(i,j))}_{t_{i,j}}\in \mathbb{C}[\bu], \\
\label{for:tableau_V}
  S^{(r)}_{\lambda/\mu}(V_M(\bv))
&=\sum_{(t_{i,j})\in \SSYT_M(\lambda'/\mu')}\prod_{(i,j)\in D(\lambda'/\mu')}v^{(r-c(i,j))}_{t_{i,j}}\in \mathbb{C}[\bv].
\end{align}
\end{theorem}

This clearly shows that 
$S^{\FK}_{\lambda/\mu}(W)$ 
is obtained as a specialization of $S^{(r)}_{\lambda/\mu}(X)$.

\begin{lemma}
\label{lemma:nnsy to fk}
For $W=\{w_{k,c}\}_{k\in [M],c\in\mathbb{Z}}$,
define 
$\bu=\{u^{(t)}_{k}\}_{k\in[M],\,t\in [N-1]}$ and 
$\bv=\{v^{(t)}_{k}\}_{k\in[M],\,t\in [N-1]}$
by $u^{(t)}_k=w_{k,t-r}$ and 
$v^{(t)}_k=w_{k,-t+r}$, respectively.
Then, for $m\in\mathbb{Z}$, we have  
\begin{align}
\label{for:nnsy to fk by U}
 S^{\FK}_{\lambda/\mu}(\tau^{m}W)
&=S^{(r+m)}_{\lambda/\mu}(U_M(\bu)),\\
\label{for:nnsy to fk by V}
 S^{\FK}_{\lambda/\mu}(\tau^{-m}W)
&=S^{(r+m)}_{\lambda'/\mu}(V_M(\bu)),
\end{align}
where $\tau^m W:=\{w_{k,m+c}\}_{k\in [M],\,c\in\mathbb{Z}}$.
In particular, for $d\in\mathbb{Z}$, 
\begin{align*}
h^{\FK}_d(\tau^{m}W)=h^{(r+m)}_d(U_M(\bu)), \quad  
e^{\FK}_d(\tau^{m}W)=e^{(r+m)}_d(U_M(\bu)).
\end{align*}
\end{lemma}

\begin{remark}
From 
Theorem~\ref{thm:JT-nnsy},
Theorem~\ref{thm:Giambelli nnsy}
and Lemma~\ref{lemma:nnsy to fk},
one can immediately deduce the Jacobi-Trudi and the dual Jacobi-Trudi formulas 
\begin{align*}
 S^{\FK}_{\lambda/\mu}(\tau^m W)
&=\det\left[h^{\FK}_{\lambda_i-\mu_j-i+j}(\tau^{m+\mu_j-j+1}W)\right]_{1\le i,j\le \ell(\lambda)},\\
 S^{\FK}_{\lambda/\mu}(\tau^{-m} W)
&=\det\left[e^{\FK}_{\lambda'_i-\mu'_j-i+j}(\tau^{m-\mu_j+j-1}W)\right]_{1\le i,j\le \ell(\lambda')},
\end{align*}
obtained in \cite[Corollary~5.1 and Corollary~5.2]{fk21}, respectively, and the Giambelli formula 
\begin{align*}
 S^{\FK}_{\lambda/\mu}(\tau^m W)
=(-1)^q\det\left[
\begin{array}{c|c}
\left[S^{\FK}_{(\alpha_i\,|\,\beta_j)}(\tau^m W)\right]_{\substack{1\le i\le p \\ 1\le j\le p}} & 
\left[h^{\FK}_{\alpha_i-\gamma_j}(\tau^{m+\gamma_j+1} W)\right]_{\substack{1\le i\le p \\ 1\le j\le q}} \\[10pt]
\hline \\[-10pt]
 \left[e^{\FK}_{\beta_j-\delta_i}(\tau^{m-\delta_i-1}W)\right]_{\substack{1\le i\le q \\ 1\le j\le p}} & O_{q}
\end{array}
\right],
\end{align*}
obtained in \cite[Corollary~5.3]{fk21}.
\end{remark}



\section{Applications of the Desnanot-Jacobi's adjoint matrix theorem
}
\label{djformulasection}

In this section, applying the following Desnanot-Jacobi's adjoint matrix theorem, 
we derive some algebraic relations among $S^{(r)}_{\lambda/\mu}(X)$.

\begin{theorem}
[Desnanot-Jacobi's adjoint matrix theorem]
\label{djformula}
For an arbitrary $(k+1)$-by-$(k+1)$ matrix $Z$, we have 
\begin{align}
\label{for:Desnanot-Jacobi formula}
\xi^{1,\ldots,k+1}_{1,\ldots,k+1}(Z)
\xi^{2,\ldots,k}_{2,\ldots,k}(Z)
=
\xi^{1,\ldots,k}_{1,\ldots,k}(Z)
\xi^{2,\ldots,k+1}_{2,\ldots,k+1}(Z)
-
\xi^{2,\ldots,k+1}_{1,\ldots,k}(Z)
\xi^{1,\ldots,k}_{2,\ldots,k+1}(Z).
\end{align}
\end{theorem}

This formula can be illustrated by the following diagrams when $k=3$.
\[
\ytableausetup{boxsize=8pt,aligntableaux=center} 
\begin{ytableau}
 *(gray) & *(gray) & *(gray) & *(gray) \\
 *(gray) & *(gray) & *(gray) & *(gray) \\
 *(gray) & *(gray) & *(gray) & *(gray) \\
 *(gray) & *(gray) & *(gray) & *(gray) \\
\end{ytableau}
\cdot 
\begin{ytableau}
 \, & \, & \, & \, \\
 \, & *(gray) & *(gray) & \, \\
 \, & *(gray) & *(gray) & \, \\
 \, & \, & \, & \, 
\end{ytableau}
=
\begin{ytableau}
 *(gray) & *(gray) & *(gray) & \, \\
 *(gray) & *(gray) & *(gray) & \, \\
 *(gray) & *(gray) & *(gray) & \, \\
 \, & \, & \, & \, 
\end{ytableau}
\cdot 
\begin{ytableau}
 \, & \, & \, & \, \\
 \, & *(gray) & *(gray) & *(gray) \\
 \, & *(gray) & *(gray) & *(gray) \\
 \, & *(gray) & *(gray) & *(gray)  
\end{ytableau}
-
\begin{ytableau}
 \, & \, & \, & \, \\
 *(gray) & *(gray) & *(gray) & \, \\
 *(gray) & *(gray) & *(gray) & \, \\
 *(gray) & *(gray) & *(gray) & \, \\
\end{ytableau}
\cdot 
\begin{ytableau}
 \, & *(gray) & *(gray) & *(gray) \\
 \, & *(gray) & *(gray) & *(gray) \\
 \, & *(gray) & *(gray) & *(gray) \\
 \, & \, & \, & \, 
\end{ytableau}
\,.
\]
Using Theorem~\ref{djformula}, 
Fulmek and Kleber \cite[Theorem 2]{fk} proved the identity 
\begin{align}
\label{for:Fulmek and Kleber DJAMT}
s_{(\lambda_1,\ldots,\lambda_{k})}s_{(\lambda_2,\ldots,\lambda_{k+1})}
=s_{(\lambda_2,\ldots,\lambda_{k})}s_{(\lambda_1,\ldots,\lambda_{k+1})}
+s_{(\lambda_2-1,\ldots,\lambda_{k+1}-1)}s_{(\lambda_1+1,\ldots,\lambda_{k}+1)}.
\end{align}
We first generalize \eqref{for:Fulmek and Kleber DJAMT} to $S^{(r)}_{\lambda/\mu}(X)$.

\begin{theorem}
\label{fk9th}
Let $\lambda/\mu$ be a skew partition with 
$\lambda=(\lambda_1,\ldots,\lambda_{k+1})$ and 
$\mu=(\mu_1,\ldots,\mu_{k+1})$,
and $\lambda'=(\lambda'_1,\ldots,\lambda'_{l+1})$ and 
$\mu'=(\mu'_1,\ldots,\mu'_{l+1})$ their conjugates, respectively.
\begin{itemize}
\item[(1)] 
 It holds that 
\begin{align}
\label{for:DJAMT H}
\begin{split}
& 
S_{(\lambda_1,\ldots,\lambda_{k+1})/(\mu_1,\ldots,\mu_{k+1})}^{(r)}(X)
\cdot S_{(\lambda_2,\ldots,\lambda_{k})/(\mu_2,\ldots,\mu_{k})}^{(r-1)}(X)
\\
&=S_{(\lambda_1,\ldots,\lambda_{k})/(\mu_1,\ldots,\mu_{k})}^{(r)}(X)
\cdot S_{(\lambda_2,\ldots,\lambda_{k+1})/(\mu_2,\ldots,\mu_{k+1})}^{(r-1)}(X)
\\
&\ \ \ 
-S_{(\lambda_2-1,\ldots,\lambda_{k+1}-1)/(\mu_1,\ldots,\mu_{k})}^{(r)}(X)
\cdot S_{(\lambda_1+1,\ldots,\lambda_{k}+1)/(\mu_2,\ldots,\mu_{k+1})}^{(r-1)}(X).
\end{split}
\end{align}
\item[(2)]  
 It holds that 
\begin{align}
\label{for:DJAMT E}
\begin{split}
& 
S_{(\lambda'_1,\ldots,\lambda'_{l+1})'/(\mu'_1,\ldots,\mu'_{l+1})'}^{(r)}(X)
 \cdot S_{(\lambda'_2,\ldots,\lambda'_{l})'/(\mu'_2,\ldots,\mu'_{l})'}^{(r+1)}(X)
\\
&=S_{(\lambda'_1,\ldots,\lambda'_{l})'/(\mu'_1,\ldots,\mu'_{l})'}^{(r)}(X)
 \cdot S_{(\lambda'_2,\ldots,\lambda'_{l+1})'/(\mu'_2,\ldots,\mu'_{l+1})'}^{(r+1)}(X)\\
&\ \ \ 
-S_{(\lambda'_2-1,\ldots,\lambda'_{l+1}-1)'/(\mu'_1,\ldots,\mu'_{l})'}^{(r)}(X)
 \cdot  S_{(\lambda'_1+1,\ldots,\lambda'_{l}+1)'/(\mu'_2,\ldots,\mu'_{l+1})'}^{(r+1)}(X).
 \end{split}
\end{align}
\end{itemize}
Here, we understand $S^{(r)}_{\lambda/\mu}(X)=0$
if $\lambda/\mu$ is not a skew partition.
\end{theorem}
\begin{proof}
Applying \eqref{for:Desnanot-Jacobi formula} 
to the case $Z=H^{(r)}_{\lambda/\mu}(X)
 $ with 
\begin{align*}
\xi^{1,\ldots,k+1}_{1,\ldots,k+1}(Z)
&=S^{(r)}_{(\lambda_1,\ldots,\lambda_{k+1})/(\mu_1,\ldots,\mu_{k+1})}(X),
&
\xi^{2,\ldots,k}_{2,\ldots,k}(Z)
&=S^{(r-1)}_{(\lambda_2,\ldots,\lambda_{k})/(\mu_2,\ldots,\mu_{k})}(X),\\
\xi^{1,\ldots,k}_{1,\ldots,k}(Z)
&=S^{(r)}_{(\lambda_1,\ldots,\lambda_{k})/(\mu_1,\ldots,\mu_{k})}(X),
&
\xi^{2,\ldots,k+1}_{2,\ldots,k+1}(Z)
&=S^{(r-1)}_{(\lambda_2,\ldots,\lambda_{k+1})/(\mu_2,\ldots,\mu_{k+1})}(X),\\
\xi^{2,\ldots,k+1}_{1,\ldots,k}(Z)
&=S^{(r)}_{(\lambda_2-1,\ldots,\lambda_{k+1}-1)/(\mu_1,\ldots,\mu_{k})}(X),
&
\xi^{1,\ldots,k}_{2,\ldots,k+1}(Z)
&=S^{(r-1)}_{(\lambda_1+1,\ldots,\lambda_{k}+1)/(\mu_2,\ldots,\mu_{k+1})}(X),
\end{align*}
which are derived from \eqref{for:9thJT-nnsy},
we obtain the first assertion.
One can similarly prove the second one by considering the cases $Z=E^{(r)}_{\lambda/\mu}(X)
$ with \eqref{for:9thDJT-nnsy}.
\end{proof}


\begin{example}
When $\lambda/\mu
=(5,4,4,3)/(3,1,1)$,
 we have 
\begin{align*}
\ytableausetup{boxsize=10pt,aligntableaux=center} 
\begin{ytableau}
  \none &  \none & \none & {\scriptstyle 3} & {\scriptstyle 4}  \\
  \none &  {\scriptstyle 0} & {\scriptstyle 1} &  {\scriptstyle 2}  \\
 \none &  {\scriptstyle \ol1} &  {\scriptstyle 0} &  {\scriptstyle 1} \\
  {\scriptstyle \ol3} &  {\scriptstyle \ol2} & {\scriptstyle \ol1} 
\end{ytableau}
\cdot 
\begin{ytableau}
  {\scriptstyle 0} & {\scriptstyle 1} &  {\scriptstyle 2}  \\
   {\scriptstyle \ol1} &  {\scriptstyle 0} &  {\scriptstyle 1} 
\end{ytableau}
&\,\overset{\eqref{for:DJAMT H}}{=}\,
\begin{ytableau}
  \none &  \none & \none & {\scriptstyle 3} & {\scriptstyle 4}  \\
  \none &  {\scriptstyle 0} & {\scriptstyle 1} &  {\scriptstyle 2}  \\
 \none &  {\scriptstyle \ol1} &  {\scriptstyle 0} &  {\scriptstyle 1} 
\end{ytableau}
\cdot
\begin{ytableau}
  \none &  {\scriptstyle 0} & {\scriptstyle 1} &  {\scriptstyle 2}  \\
 \none &  {\scriptstyle \ol1} &  {\scriptstyle 0} &  {\scriptstyle 1} \\
  {\scriptstyle \ol3} &  {\scriptstyle \ol2} & {\scriptstyle \ol1} 
\end{ytableau}
-\,
\begin{ytableau}
  {\scriptstyle 0} & {\scriptstyle 1}   \\
  {\scriptstyle \ol1}  
\end{ytableau}
\cdot
\begin{ytableau}
  \none &  {\scriptstyle 0} & {\scriptstyle 1} & {\scriptstyle 2} & {\scriptstyle 3} & {\scriptstyle 4} \\
  \none &  {\scriptstyle \ol1} & {\scriptstyle 0} &  {\scriptstyle 1} & {\scriptstyle 2} \\
 {\scriptstyle \ol3} &  {\scriptstyle \ol2} &  {\scriptstyle \ol1} &  {\scriptstyle 0} & {\scriptstyle 1}
\end{ytableau}
\,,\\
\begin{ytableau}
  \none &  \none & \none & {\scriptstyle 3} & {\scriptstyle 4}  \\
  \none &  {\scriptstyle 0} & {\scriptstyle 1} &  {\scriptstyle 2}  \\
 \none &  {\scriptstyle \ol1} &  {\scriptstyle 0} &  {\scriptstyle 1} \\
  {\scriptstyle \ol3} &  {\scriptstyle \ol2} & {\scriptstyle \ol1} 
\end{ytableau}
\cdot 
\begin{ytableau}
  \none &  \none & {\scriptstyle 3}  \\
   {\scriptstyle 0} & {\scriptstyle 1} &  {\scriptstyle 2}  \\
   {\scriptstyle \ol1} & {\scriptstyle 0} &  {\scriptstyle 1}  \\
  {\scriptstyle \ol2} &  {\scriptstyle \ol1}  
\end{ytableau}
&\,\overset{\eqref{for:DJAMT E}}{=}\,
\begin{ytableau}
  \none &  \none & \none & {\scriptstyle 3}  \\
  \none &  {\scriptstyle 0} & {\scriptstyle 1} &  {\scriptstyle 2}  \\
 \none &  {\scriptstyle \ol1} &  {\scriptstyle 0} &  {\scriptstyle 1} \\
 {\scriptstyle \ol3} &  {\scriptstyle \ol2} &  {\scriptstyle \ol1} 
\end{ytableau}
\cdot
\begin{ytableau}
  \none &  \none & {\scriptstyle 3} &  {\scriptstyle 4}  \\
  {\scriptstyle 0} & {\scriptstyle 1} &  {\scriptstyle 2}  \\
  {\scriptstyle \ol1} &  {\scriptstyle 0} &  {\scriptstyle 1} \\
  {\scriptstyle \ol2} & {\scriptstyle \ol1} 
\end{ytableau}
-\,
\begin{ytableau}
  {\scriptstyle 0} & {\scriptstyle 1}   \\
  {\scriptstyle \ol1}  
\end{ytableau}
\cdot
\begin{ytableau}
  \none &  \none & {\scriptstyle 3} & {\scriptstyle 4}  \\
  {\scriptstyle 0} & {\scriptstyle 1} &  {\scriptstyle 2} & {\scriptstyle 3} \\
  {\scriptstyle \ol1} & {\scriptstyle 0} &  {\scriptstyle 1} & {\scriptstyle 2} \\
  {\scriptstyle \ol2} & {\scriptstyle \ol1} &  {\scriptstyle 0} & {\scriptstyle 1} \\
  {\scriptstyle \ol3} &  {\scriptstyle \ol2} &  {\scriptstyle \ol1} 
\end{ytableau}
\,.
\end{align*}
\end{example}


Let $[\,m\,|\,n\,]$ denote the $m$-by-$n$ rectangle $(n^m)$.
Applying \eqref{for:DJAMT H} (or \eqref{for:DJAMT E}) to rectangular shapes, 
one obtain the following relations.

\begin{cor}
\label{square}
For $p,q\ge 1$, we have 
\begin{align}
\label{for:9thRectangle}
\begin{split}
S_{[\,p+1\,|\,q\,]}^{(r)}(X)
S_{[\,p-1\,|\,q\,]}^{(r-1)}(X)
=S_{[\,p\,|\,q\,]}^{(r)}(X)
S_{[\,p\,|\,q\,]}^{(r-1)}(X)
-S_{[\,p\,|\,q-1\,]}^{(r)}(X)
S_{[\,p\,|\,q+1\,]}^{(r-1)}(X).
\end{split}
\end{align}
\end{cor}

We next give algebraic relations among $S_{\lambda/\mu}^{(r)}(X)$ from the Giambelli formula (Theorem~\ref{thm:Giambelli nnsy}).
To state the result, we prepare some notations. For a non-empty tuple $\alpha=(\alpha_1,\ldots,\alpha_p)$, put 
${}^{-}\alpha=(\alpha_2,\ldots,\alpha_p)$, 
$\alpha^{-}=(\alpha_1,\ldots,\alpha_{p-1})$
and 
${}^{-}\alpha^{-}=(\alpha_2,\ldots,\alpha_{p-1})$.
Moreover, for $1\le j\le p$, put $\xi^{\vee}(j)=(\xi_1,\ldots,\xi_{j-1},\xi_{j+1},\ldots,\xi_p)$.

\begin{theorem}
Let $\lambda/\mu$ be a skew partition.
Write $\lambda=(\alpha\,|\,\beta)$
and $\mu=(\gamma\,|\delta)$
in the Frobenius notations, where 
$\alpha=(\alpha_1,\ldots,\alpha_p)$,
$\beta=(\beta_1,\ldots,\beta_p)$,
$\gamma=(\gamma_1,\ldots,\gamma_q)$ 
and $\delta=(\delta_1,\ldots,\delta_q)$.
\begin{itemize}
\item[(1)] 
It holds that 
\begin{align}
\label{for:non-skew Giambelli}
\begin{split}
& S_{(\alpha\,|\,\beta)}^{(r)}(X)
 S_{({}^{-}\alpha^{-}\,|\,{}^{-}\beta^{-})}^{(r)}(X)
 =S_{(\alpha^{-}\,|\,\beta^{-})}^{(r)}(X)
 S_{({}^{-}\alpha\,|\,{}^{-}\beta)}^{(r)}(X)
-S_{({}^{-}\alpha\,|\,\beta^{-})}^{(r)}(X)
 S_{(\alpha^{-}\,|\,{}^{-}\beta)}^{(r)}(X).
\end{split}
\end{align}
\item[(2)] 
It holds that 
\begin{align}
\label{for:skew Giambelli}
\begin{split}
&S_{(\alpha\,|\,\beta)/(\gamma\,|\,\delta)}^{(r)}(X)
S_{({}^{-}\alpha\,|\,{}^{-}\beta)/(\gamma^{-}\,|\,\delta^{-})}^{(r)}(X)\\
&
=S_{(\alpha\,|\,\beta)/(\gamma^{-}\,|\,\delta^{-})}^{(r)}(X)
 S_{({}^{-}\alpha\,|\,{}^{-}\beta)/(\gamma\,|\,\delta)}^{(r)}(X)\\
 &\hspace{5mm}+\sum^{p}_{i,j=1}(-1)^{i+j}h_{\alpha_i-\gamma_q}^{(r+\gamma_q+1)}(X)e_{\beta_j-\delta_q}^{(r-\delta_q-1)}(X) S_{(\alpha^{\vee}(i)\,|\,{}^{-}\beta)/(\gamma^{-}\,|\,\delta^{-})}^{(r)}(X)S_{({}^{-}\alpha\,|\,\beta^{\vee}(j))/(\gamma^{-}\,|\,\delta^{-})}^{(r)}(X).
\end{split}
\end{align}    
\end{itemize}
\end{theorem}
\begin{proof}
These formulas are also direct consequences of Theorem~\ref{djformula}.
Actually, applying \eqref{for:Desnanot-Jacobi formula} 
to the case $Z=G^{(r)}_{\lambda/\mu}(X)$ together with the following expressions of minor determinants derived from Theorem~\ref{thm:Giambelli nnsy},
we have the desired formulas:
\begin{align*}
\xi^{1,\ldots,p+q}_{1,\ldots,p+q}(Z)
&=(-1)^qS_{(\alpha\,|\,\beta)/(\gamma\,|\,\delta)}(X),\\
\xi^{2,\ldots,p+q-1}_{2,\ldots,p+q-1}(Z)
&=
\begin{cases}
S_{({}^{-}\alpha^{-}\,|\,{}^{-}\beta^{-})}^{(r)}(X)
& (\mu=\varnothing),\\
(-1)^{q-1}S_{({}^{-}\alpha\,|\,{}^{-}\beta)/(\gamma^{-}\,|\,\delta^{-})}^{(r)}(X) & (\mu\ne \varnothing),
\end{cases}
\\
\xi^{1,\ldots,p+q-1}_{1,\ldots,p+q-1}(Z)
&=
\begin{cases}
S_{(\alpha^{-}\,|\,\beta^{-})}^{(r)}(X) & (\mu=\varnothing),\\
(-1)^{q-1}S_{(\alpha\,|\,\beta)/(\gamma^{-}\,|\,\delta^{-})}^{(r)}(X) & (\mu\ne \varnothing),\\
\end{cases}
\\
\xi^{2,\ldots,p+q}_{2,\ldots,p+q}(Z)
&=(-1)^qS_{({}^{-}\alpha\,|\,{}^{-}\beta)/(\gamma\,|\,\delta)}^{(r)}(X),\\
\xi^{2,\ldots,p+q}_{1,\ldots,p+q-1}(Z)
&=\begin{cases}
S_{({}^{-}\alpha\,|\,\beta^{-})}^{(r)}(X) & (\mu=\varnothing),\\[5pt]
\displaystyle{\sum^{p}_{j=1}(-1)^{p+j}e_{\beta_j-\delta_q}^{(r-\delta_q-1)}(X)S_{({}^{-}\alpha\,|\,\beta^{\vee}(j))/(\gamma^{-}\,|\,\delta^{-})}^{(r)}(X)} & (\mu\ne \varnothing),
\end{cases}
\\
\xi^{1,\ldots,p+q-1}_{2,\ldots,p+q}(Z)
&=\begin{cases}
S_{(\alpha^{-}\,|\,{}^{-}\beta)}^{(r)}(X) & (\mu=\varnothing),\\[5pt]
\displaystyle{\sum^{p}_{i=1}(-1)^{p+i}h_{\alpha_i-\gamma_q}^{(r+\gamma_q+1)}(X)S_{(\alpha^{\vee}(i)\,|\,{}^{-}\beta)/(\gamma^{-}\,|\,\delta^{-})}^{(r)}(X)} & (\mu\ne \varnothing).
\end{cases}
\end{align*}
Notice that,
in the last two cases with $\mu\ne\varnothing$, 
we have used cofactor expansions.
\end{proof}

\begin{example}
When $\lambda=(5,5,5,3)=(4,3,2\,|\,3,2,1)$, we have from \eqref{for:non-skew Giambelli}
\begin{align*}
\ytableausetup{boxsize=12pt,aligntableaux=center}
\scriptsize
\begin{ytableau}
0&1&2&3&4\\
\mb{1} & 0 & 1 & 2 &3\\
\mb{2} & \mb{1} & 0 &1 &2\\
\mb{3} & \mb{2} & \mb{1} 
\end{ytableau}
\cdot 
\begin{ytableau}
0 & 1 & 2 & 3\\
\mb{1} \\
\mb{2}
\end{ytableau}
\,=\,
\begin{ytableau}
0&1&2&3&4\\
\mb{1} & 0 & 1 & 2 & 3\\
\mb{2} & \mb{1}  \\
\mb{3} & \mb{2} 
\end{ytableau}
\cdot 
\begin{ytableau}
0 & 1 & 2 & 3\\
 \mb{1} & 0 &1 & 2\\
\mb{2} & \mb{1} \\
\end{ytableau}
-
\begin{ytableau}
0&1&2&3\\
\mb{1} & 0 &1&2 \\
\mb{2} & \mb{1} \\
\mb{3} & \mb{2}  
\end{ytableau}
\cdot 
\begin{ytableau}
0&1&2&3&4\\
\mb{1} & 0 & 1 & 2 & 3\\
\mb{2} & \mb{1}\\
\end{ytableau}
\,.
\end{align*}
Moreover, when $\lambda/\mu=(5,5,5,3)/(3,2)=(4,3,2\,|\,3,2,1)/(2,0\,|\,1,0)$, 
we have from \eqref{for:skew Giambelli} 
\begin{align*}
\ytableausetup{boxsize=12pt,aligntableaux=center}
&
{\scriptsize
\begin{ytableau}
\none &\none &\none &3&4\\
\none & \none & 1 & 2 &3\\
\mb{2} & \mb{1} & 0 &1 &2\\
\mb{3} & \mb{2} & \mb{1}  
\end{ytableau}
}
\cdot 
{\scriptsize
\begin{ytableau}
\none & \none & \none & 3\\
\none & 0 & 1 & 2 \\
\mb{2} & \mb{1}
\end{ytableau}
}
\,=\,
{\scriptsize
\begin{ytableau}
\none &\none &\none &3&4\\
\none & 0 & 1 & 2 & 3\\
\mb{2} & \mb{1} & 0 & 1 & 2 \\
\mb{3} & \mb{2} & \mb{1} 
\end{ytableau}
}
\cdot 
{\scriptsize
\begin{ytableau}
\none & \none & \none & 3\\
\none & \none &1 & 2\\
\mb{2} & \mb{1} \\
\end{ytableau}
}
+\sum^{3}_{i,j=1}(-1)^{i+j}g_{ij},
\\
\end{align*}
where
\begin{align*}
\ytableausetup{boxsize=12pt,aligntableaux=center}
g_{11}
&={\scriptsize
\begin{ytableau}
1 & 2 & 3 & 4\\
\end{ytableau}
}
\cdot 
{\scriptsize
\begin{ytableau}
\mb{1} \\
\mb{2} \\
\mb{3}
\end{ytableau}
}
\cdot 
{\scriptsize
\begin{ytableau}
\none & \none & \none & 3\\
\none & 0 &1 & 2\\
\mb{2} & \mb{1} \\
\end{ytableau}
}
\cdot 
{\scriptsize
\begin{ytableau}
\none & \none & \none & 3\\
\none & 0 &1 & 2\\
\mb{2} & \mb{1} \\
\end{ytableau}
}\,,&
g_{12}
&={\scriptsize
\begin{ytableau}
1 & 2 & 3 & 4\\
\end{ytableau}
}
\cdot 
{\scriptsize
\begin{ytableau}
\mb{1} \\
\mb{2} 
\end{ytableau}
}
\cdot 
{\scriptsize
\begin{ytableau}
\none & \none & \none & 3\\
\none & 0 &1 & 2\\
\mb{2} & \mb{1} \\
\end{ytableau}
}
\cdot 
{\scriptsize
\begin{ytableau}
\none & \none & \none & 3\\
\none & 0 &1 & 2\\
\mb{2} & \mb{1} \\
\mb{3} 
\end{ytableau}
}\,,\\
g_{13}
&={\scriptsize
\begin{ytableau}
1 & 2 & 3 & 4\\
\end{ytableau}
}
\cdot 
{\scriptsize
\begin{ytableau}
\mb{1} 
\end{ytableau}
}
\cdot 
{\scriptsize
\begin{ytableau}
\none & \none & \none & 3\\
\none & 0 &1 & 2\\
\mb{2} & \mb{1} \\
\end{ytableau}
}
\cdot 
{\scriptsize
\begin{ytableau}
\none & \none & \none & 3\\
\none & 0 &1 & 2\\
\mb{2} & \mb{1} \\
\mb{3} & \mb{2}
\end{ytableau}
}\,,\\
g_{21}
&={\scriptsize
\begin{ytableau}
1 & 2 & 3 \\
\end{ytableau}
}
\cdot 
{\scriptsize
\begin{ytableau}
\mb{1} \\
\mb{2} \\
\mb{3} 
\end{ytableau}
}
\cdot 
{\scriptsize
\begin{ytableau}
\none & \none & \none & 3&4\\
\none & 0 &1 & 2\\
\mb{2} & \mb{1} \\
\end{ytableau}
}
\cdot 
{\scriptsize
\begin{ytableau}
\none & \none & \none & 3\\
\none & 0 &1 & 2\\
\mb{2} & \mb{1} \\
\end{ytableau}
}\,,&
g_{22}
&={\scriptsize
\begin{ytableau}
1 & 2 & 3 \\
\end{ytableau}
}
\cdot 
{\scriptsize
\begin{ytableau}
\mb{1} \\
\mb{2} 
\end{ytableau}
}
\cdot 
{\scriptsize
\begin{ytableau}
\none & \none & \none & 3&4\\
\none & 0 &1 & 2\\
\mb{2} & \mb{1} \\
\end{ytableau}
}
\cdot 
{\scriptsize
\begin{ytableau}
\none & \none & \none & 3\\
\none & 0 &1 & 2\\
\mb{2} & \mb{1} \\
\mb{3} 
\end{ytableau}
}\,,\\
g_{23}
&={\scriptsize
\begin{ytableau}
1 & 2 & 3 \\
\end{ytableau}
}
\cdot 
{\scriptsize
\begin{ytableau}
\mb{1} 
\end{ytableau}
}
\cdot 
{\scriptsize
\begin{ytableau}
\none & \none & \none & 3&4\\
\none & 0 &1 & 2\\
\mb{2} & \mb{1} \\
\end{ytableau}
}
\cdot 
{\scriptsize
\begin{ytableau}
\none & \none & \none & 3\\
\none & 0 &1 & 2\\
\mb{2} & \mb{1} \\
\mb{3} & \mb{2} 
\end{ytableau}
}\,,\\
g_{31}
&={\scriptsize
\begin{ytableau}
1 & 2 
\end{ytableau}
}
\cdot 
{\scriptsize
\begin{ytableau}
\mb{1} \\
\mb{2} \\
\mb{3} 
\end{ytableau}
}
\cdot 
{\scriptsize
\begin{ytableau}
\none & \none & \none & 3&4\\
\none & 0 &1 & 2 &3\\
\mb{2} & \mb{1} \\
\end{ytableau}
}
\cdot 
{\scriptsize
\begin{ytableau}
\none & \none & \none & 3\\
\none & 0 &1 & 2\\
\mb{2} & \mb{1} \\
\end{ytableau}
}\,,&
g_{32}
&={\scriptsize
\begin{ytableau}
1 & 2 \\
\end{ytableau}
}
\cdot 
{\scriptsize
\begin{ytableau}
\mb{1} \\
\mb{2} 
\end{ytableau}
}
\cdot 
{\scriptsize
\begin{ytableau}
\none & \none & \none & 3&4\\
\none & 0 &1 & 2&3\\
\mb{2} & \mb{1} \\
\end{ytableau}
}
\cdot 
{\scriptsize
\begin{ytableau}
\none & \none & \none & 3\\
\none & 0 &1 & 2\\
\mb{2} & \mb{1} \\
\mb{3} 
\end{ytableau}
}\,,\\
g_{33}
&={\scriptsize
\begin{ytableau}
1 & 2 \\
\end{ytableau}
}
\cdot 
{\scriptsize
\begin{ytableau}
\mb{1} 
\end{ytableau}
}
\cdot 
{\scriptsize
\begin{ytableau}
\none & \none & \none & 3&4\\
\none & 0 &1 & 2&3\\
\mb{2} & \mb{1} \\
\end{ytableau}
}
\cdot 
{\scriptsize
\begin{ytableau}
\none & \none & \none & 3\\
\none & 0 &1 & 2\\
\mb{2} & \mb{1} \\
\mb{3} & \mb{2}
\end{ytableau}
}\,.
\end{align*}
\end{example}


\section{Applications of the Pl\"{u}cker relations
}
\label{pformula}

In this section,
we derive several quadratic relations for $S^{(r)}_{\lambda}(X)$ by applying the Pl\"{u}cker relations for determinants, which are described as follows.
Let $Z$ be a $2n$-by-$n$ matrix whose rows are indexed by $1,\ldots,n,1',\ldots,n'$ and columns by $1,\ldots,n$.
For $1\leq \ell\leq n$, 
take a row index $(t_1,\ldots,t_\ell)$ satisfying 
$1\leq t_1<\cdots<t_\ell\leq n$.
Then, the Pl\"ucker relations (fixing the rows $1',\ldots,\widehat{t'_1},\ldots,\widehat{t'_\ell},\ldots,n'$)
state that
\begin{equation}
\label{for:Plucker}
\xi^{1,\ldots,n}_{1,\ldots,n}(Z)
\xi^{1',\ldots,n'}_{1,\ldots,n}(Z)
=\sum_{1\leq s_1<\cdots<s_\ell\leq n}
\sigma_{RS}(
\xi^{1,\ldots,s_1,\ldots,s_\ell,\ldots,n}_
{1,\ldots,n}(Z)
\xi^{1',\ldots,t'_1,\ldots,t'_\ell,\ldots,n'}_
{1,\ldots,n}(Z)),    
\end{equation}
where $\sigma_{RS}$ exchanges the row $s_i$ with $t'_i$ for all $1\le i\le \ell$ before evaluating the determinants. 
For example, when $n=3$, $\ell=2$ and $(t_1,t_2)=(1,3)$, we have 
\begin{align*}
\xi^{1,2,3}_{1,2,3}(Z)
\xi^{1',2',3'}_{1,2,3}(Z)
=
\xi^{1',3',3}_{1,2,3}(Z)\xi^{1,2',2}_{1,2,3}(Z)
+\xi^{1',2,3'}_{1,2,3}(Z)\xi^{1,2',3}_{1,2,3}(Z)
+\xi^{1,1',3'}_{1,2,3}(Z)\xi^{2,2',3}_{1,2,3}(Z).
\end{align*}
Remark that we could define
Pl\"ucker relations more generally, for the minors of a matrix of any size, but they would be a specialization of \eqref{for:Plucker}.



We first give a generalization of the results obtained by Kleber \cite{k01} for Schur functions. 
To state the results, 
we introduce the combinatorial terminology of Young diagrams.
For a partition $\lambda$, 
the outside border of $\lambda$ is the strip whose cells contain all the cells not in $D(\lambda)$ but immediately below and to the right of those in $D(\lambda)$.
On the other hand, the inside border of $\lambda$ is the strip whose cells contain all the right-most or the bottom-most cells in $D(\lambda)$. 
We denote the outside and inside borders of $\lambda$ by
$\OB(\lambda)$ and $\IB(\lambda)$, respectively.
For $u,v\in \OB(\lambda)$, if the diagram obtained by adding the substrip of $\OB(\lambda)$ starting from $u$ and ending at $v$ to $D(\lambda)$ is a Young diagram, then we denote the partition by $\add^{u}_{v}(\lambda)$.
Similarly, for $u,v\in \IB(\lambda)$, if the diagram obtained by removing the substrip of $\IB(\lambda)$ starting from $u$ and ending at $v$ from $D(\lambda)$ is a Young diagram, then we denote the partition by $\rem^{u}_{v}(\lambda)$. 
Assume that $\lambda$ has $n$ corners.
Then, it can be written as  
$\lambda=(m_1^{r_1}m_2^{r_2-r_1}\cdots m_n^{r_n-r_{n-1}})$
by using integers $m_1>m_2>\cdots>m_n>m_{n+1}=0$ and $0<r_1<r_2<\cdots<r_n$.
This implies that $C(\lambda)=\{(r_1,m_1),\ldots,(r_n,m_n)\}$.
For $1\le p\le q\le n$, put 
\[
\add^{p}_{q}:=\add^{(r_p+1,m_p)}_{(r_q+1,m_{q+1}+1)}, \quad  
\rem^{p}_{q}:=\rem^{(r_p,m_p)}_{(r_q,m_{q+1}+1)}.
\]
Moreover, for $1\le p_1<\cdots<p_{t}\le q_t<\cdots<q_1\le n$, 
define 
\begin{align}
\label{def:iterated add rem}
\add^{p_1,\ldots,p_t}_{q_1,\ldots,q_t}
=\add^{p_1}_{q_1}\circ \cdots \circ \add^{p_t}_{q_t},
\quad
\rem^{p_1,\ldots,p_t}_{q_1,\ldots,q_t}
=\rem^{p_1}_{q_1}\circ \cdots \circ \rem^{p_t}_{q_t}.
\end{align}

\begin{example}
Let $\lambda=(5,4,2)$. 
Then, we have 
$C(\lambda)=\{(1,5),(2,4),(3,3)\}$ and 
\begin{align*}
\OB(\lambda)
&=\{(1,6),(2,6),(2,5),(3,5),(3,4),(3,3),(4,3),(4,2),(4,1)\},\\
\IB(\lambda)
&=\{(1,5),(1,4),(2,4),(2,3),(2,2),(3,2),(3,1)\}.    
\end{align*}
For the adding and removing  operators,
we have, for example,
\begin{align*}
\ytableausetup{boxsize=8pt,aligntableaux=center}
\add^{1,2}_{3,2}
\left(\ 
\begin{ytableau}
 \ & \ & \ & \ & \ \\
 \ & \ & \ & \ \\
 \ & \ 
\end{ytableau}
\ \right)
&=
\add^{(2,5)}_{(4,1)}\left(
\add^{(3,4)}_{(3,3)}
\left(\ \ \ \ 
\begin{tikzpicture}[thick,baseline=38pt]
\draw[line width = 8.52, gray] (1.1,1.3) -- (0.5,1.3);
\end{tikzpicture}
\hspace{-38.5pt}
\begin{ytableau}
 \ & \ & \ & \ & \ \\
 \ & \ & \ & \ & \none[\star] \\
 \ & \ & \none[\bullet] & \none[\bullet] \\
 \none[\star]
\end{ytableau}
\ \right)
\right)
=
\add^{(2,5)}_{(4,1)}\left(
\begin{tikzpicture}[thick,baseline=47pt]
\draw[line width = 8.52, gray] (2.0,1.3) -- (0.5,1.3);
\draw[line width = 8.52, gray] (1.7,1.6) -- (2.0,1.6);
\draw[line width = 8.52, gray] (1.7,1.9) -- (2.0,1.9);
\end{tikzpicture}
\hspace{-47pt}
\begin{ytableau}
 \ & \ & \ & \ & \ \\
 \ & \ & \ & \ & \none[\star] \\
 \ & \ & \ & \ \\
 \none[\star]
\end{ytableau}
\ \right)
=
\begin{ytableau}
 \ & \ & \ & \ & \ \\
 \ & \ & \ & \ & \ \\
 \ & \ & \ & \ & \ \\
 \ & \ & \ & \ & \
\end{ytableau}
 \,,
\\
\rem^{1,2}_{3,2}
\left(\ 
\begin{ytableau}
 \ & \ & \ & \ & \ \\
 \ & \ & \ & \ \\
 \ & \ 
\end{ytableau}
\ \right)
&=
\rem^{(1,5)}_{(3,1)}\left(
\rem^{(2,4)}_{(2,3)}
\left(\ 
\begin{ytableau}
 \ & \ & \ & \ & \star  \\
 \ & \ & *(gray)\bullet & *(gray)\bullet \\
 \star & \ 
\end{ytableau}
\ \right)
\right)
=
\rem^{(1,4)}_{(5,1)}\left(\ 
\begin{ytableau}
 \ & *(gray) & *(gray) & *(gray) & *(gray)\star  \\
 \ & *(gray) \\
 *(gray)\star & *(gray) 
\end{ytableau}
\ \right)
=
\begin{ytableau}
 \ \\
 \ 
\end{ytableau}
 \,.
\end{align*}
\end{example}

For a partition $\lambda=(\lambda_1,\ldots,\lambda_k)$, $a\ge 0$ and $1\le \ell\le k$, 
we denote by $\lambda\pm (a^{\ell})$ the partition $\nu=(\nu_1,\ldots,\nu_k)$
given by $\nu_i=\lambda_i\pm a$ for $1\le i\le \ell$ and $\nu_i=\lambda_i$ for $i>\ell$.
Notice that $\lambda-(a^{\ell})$ is defined only if 
$\lambda_i>a$ for $1\le i\le \ell$ and $\lambda_{\ell}-a\ge \lambda_{\ell+1}$.
Now, the quadratic relations for $s_{\lambda}$ obtained by Kleber \cite{k01} are given as follows.

\begin{theorem}[{\cite[Theorem 4.2]{k01}}]
\label{thm:kPlucker}
Let $\lambda$ be a partition with $n$ corners.
Take $1\le d\le n$ and denote by $\ell$ 
the height of the $d$-th shortest column of $\lambda$.
Then, we have  
\begin{equation}
s_{\lambda}s_{\lambda}
=s_{\lambda-(1^{\ell})}s_{\lambda+(1^{\ell})}
+\sum^{\min\{d,n-d+1\}}_{t=1}(-1)^{t-1}
\sum_{\substack{1\le p_1<\cdots<p_t\le d \\ d\le q_t<\cdots<q_1\le n}}
s_{\add^{p_1,\ldots,p_t}_{q_1,\ldots,q_t}(\lambda)}
s_{\rem^{p_1,\ldots,p_t}_{q_1,\ldots,q_t}(\lambda)}.
\end{equation}
\end{theorem}

In this section, 
we generalize this theorem to $S^{(r)}_{\lambda}(X)$ as: 
 
\begin{theorem}
\label{thm:9thPlucker}
Let $\lambda$ be a partition having $n$ corners.
Take $1\le d\le n$. 
\begin{enumerate}
\item  
Denote the $d$-th shortest column height of $\lambda$ as $\ell$.
Then, we have  
\begin{align}
\label{for:9thPlucker}
\begin{split}
&S_{\lambda}^{(r)}(X)S_{\lambda}^{(r-1)}(X)
=S_{\lambda-(1^{\ell})}^{(r)}(X)
S_{\lambda+(1^{\ell})}^{(r-1)}(X)
\\
&\quad +\sum^{\min\{d,n-d+1\}}_{t=1}(-1)^{t-1}
\sum_{\substack{1\le p_1<\cdots<p_t\le d \\ d\le q_t<\cdots<q_1\le n}}
S_{\add^{p_1,\ldots,p_t}_{q_1,\ldots,q_t}(\lambda)}^{(r)}(X)
S_{\rem^{p_1,\ldots,p_t}_{q_1,\ldots,q_t}(\lambda)}^{(r-1)}(X).
\end{split}
\end{align}
\item 
Denote the $d$-th shortest row length of $\lambda$ as $\ell$.
Then, we have  
\begin{align}
\label{for:9thDualPlucker}
\begin{split}
&S_{\lambda}^{(r)}(X)S_{\lambda}^{(r+1)}(X)
=S_{(\lambda'-(1^{\ell}))'}^{(r)}(X)
S_{(\lambda'+(1^{\ell}))'}^{(r+1)}(X)\\
&\quad +\sum^{\min\{d,n-d+1\}}_{t=1}(-1)^{t-1}
\sum_{\substack{1\le p_1<\cdots<p_t\le d \\ d\le q_t<\cdots<q_1\le n}}
S_{(\add^{p_1,\ldots,p_t}_{q_1,\ldots,q_t}(\lambda'))'}^{(r)}(X)
S_{(\rem^{p_1,\ldots,p_t}_{q_1,\ldots,q_t}(\lambda'))'}^{(r+1)}(X).
\end{split}
\end{align}
\end{enumerate}
\end{theorem}

To give a proof of Theorem~\ref{thm:9thPlucker}, 
we need to construct a new matrix from two square matrices having the same size:
For matrices $A=[a_{i,j}]_{1\le i,j\le n}$ and $B=[b_{i,j}]_{1\le i,j\le n}$ of size $n$,
define the matrix $A\,\square\,B=M=[M_{i,j}]$ 
having $2n+2$ rows indexed by $i=L,R,1,\ldots,n,1',\ldots,n'$, where we understand $R,L<1$, 
and $n+1$ columns indexed by $j=1,\ldots,n+1$ by 
\begin{align*}
M_{L,j}
&=\delta_{1,j},&
M_{R,j}
&=(-1)^n\delta_{n+1,j},&
M_{i,j}
&=a_{i,j} \ \ (1\le i\le n),&
M_{i',j}
&=b_{i,j-1} \ \ (1\le i\le n),
\end{align*}
where $a_{i,n+1}$ and $b_{i,0}$ are naturally defined from $A$ and $B$, respectively.
For example, 
\[
\left[
\begin{array}{cc}
a & b \\ c & d
\end{array}
\right]
\square 
\left[
\begin{array}{cc}
x & y \\ z & w
\end{array}
\right]
=
\begin{blockarray}{cccc}
\ & 1 & 2 & 3 \\
\begin{block}{c[ccc]}
  L & 1 & 0 & 0 \\
  \cline{2-4}
  R & 0 & 0 & (-1)^2 \\
  \cline{2-4}
  1 & a & b & \ast \\
  2 & c & d & \ast \\
  \cline{2-4}
  1' & \ast & x & y \\
  2' & \ast & z & w \\
\end{block}
\end{blockarray}
\,.
\]

\begin{proof}[Proof of Theorem~\ref{thm:9thPlucker}]
We mimic the proof of \cite[Theorem 4.2]{k01}.

Write $\lambda=(\lambda_1,\ldots,\lambda_{r_n})=(m_1^{r_1}\cdots m_d^{r_d-r_{d-1}}\cdots m_n^{r_n-r_{n-1}})$ as above, and $\rho=r_n$.
Note that in this setting $\ell=r_d$.
For $a,b\ge 0$, 
let $M=[M_{i,j}]=H^{(r)}_{\lambda-(a^{\ell})}(X)\,\square\,H^{(r-1)}_{\lambda+(b^{\ell})}(X)$, where $H^{(r)}_{\lambda}(X)$ is defined in Theorem \ref{thm:JT-nnsy}.
We have 
$M_{L,j}=\delta_{1,j}$,
$M_{R,j}=(-1)^{\rho}\delta_{\rho+1,j}$ and 
\begin{align*}
M_{i,j}
=
\begin{cases}
h_{\lambda_i-a-i+j}^{(r-j+1)}(X) & (1\le i\le \ell),\\
h_{\lambda_i-i+j}^{(r-j+1)}(X) & (\ell< i\le {\rho}),
\end{cases}
\quad 
M_{i',j}
=
\begin{cases}
h_{\lambda_i+b-i+(j-1)}^{(r-1-(j-1)+1)}(X) & (1\le i\le \ell),\\
h_{\lambda_i-i+(j-1)}^{(r-1-(j-1)+1)}(X) & (\ell< i\le {\rho}).
\end{cases}
\end{align*}
From now on, for simplicity, write 
$[i_0,i_1,\ldots,i_{\rho}]=\xi^{i_0,i_1,\ldots,i_{\rho}}_{1,\ldots,\rho+1}(M)$.
Applying the Pl\"{u}cker relations fixing the rows $1',2',\ldots,\ell'$ to $M$
together with the identities  
\begin{align*}
[R,1,\ldots ,\ell,\ell+1,\ldots,{\rho}]
[L,1',\ldots ,\ell',(\ell+1)',\ldots,\rho']
&=S_{\lambda-(a^{\ell})}^{(r)}(X)
S_{\lambda+(b^{\ell})}^{(r-1)}(X),\\
[L,1,\ldots,\ell,(\ell+1)',\ldots,\rho']
[R,1',\ldots ,\ell',\ell+1,\ldots,\rho]
&=S_{\lambda-(a-1)^{\ell}}^{(r-1)}(X)
S_{\lambda+(b-1)^{\ell}}^{(r)}(X)
\end{align*}
obtained from Theorem~\ref{thm:JT-nnsy}, 
we have  
\begin{align}
\label{for:Plucker first step}
 S_{\lambda-(a^{\ell})}^{(r)}(X)
S_{\lambda+(b^{\ell})}^{(r-1)}(X) 
&=S_{\lambda+((b-1)^{\ell})}^{(r)}(X)
S_{\lambda-((a-1)^{\ell})}^{(r-1)}(X)
+\sum_{\substack{\sigma=(s_L,s_{\ell+1},\cdots,s_{\rho}) \\
R\le s_L<s_{\ell+1}<\cdots<s_{\rho}\le \rho \\
\sigma \ne (R,\ell+1,\ldots,{\rho})}}
A_{\sigma}B_{\sigma},
\end{align}
where 
$A_{\sigma}:=[a_R,a_1,\ldots,a_{\rho}]$
and $B_{\sigma}:=[b_R,b_1,\ldots,b_{\rho}]$ 
are respectively defined by 
\[
a_i
=\begin{cases}
L & (i=s_L),\\    
j' & (\text{$i=s_j$ for some $\ell+1\le j\le \rho$}),\\
i & (\text{otherwise}),
\end{cases}
\quad 
b_i
=\begin{cases}
i' & (1\le i\le \ell),\\
s_i & (\text{otherwise}).
\end{cases}
\]
Notice that $s_L=R$ or $s_L\le\ell$,
and $s_j\le j$ for $\ell+1\le j\le \rho$. 

Now, take $a=b=1$.
We first observe when $A_{\sigma}B_{\sigma}$ vanishes.
\begin{itemize}
\item When $\lambda_i>\lambda_{i+1}=\cdots=\lambda_{i+u}>\lambda_{i+u+1}$ for some $i$ and $u\ge 2$ with $i+u\le \ell$,
we see that $M_{k,j}=M_{(k+1)',j}$ for $i+1\le k\le i+u-1$ and $1\le j\le \rho+1$. 
This means that $B_{\sigma}=0$ if  
$s_j\in\{i+1,\ldots,i+u-1\}$ for some $j\in\{L,\ell+1,\ldots,\rho\}$.

\item When  $\lambda_i>\lambda_{i+1}=\cdots=\lambda_{i+v}>\lambda_{i+v+1}$ for some $i$ and $v\ge 2$ with $i\ge \ell$,
we see that 
$M_{k,j}=M_{(k-1)',j}$ for $i+2\le k\le i+v$ and $1\le j\le \rho+1$. 
This means that 
$A_{\sigma}=0$ if $\{i+2,\ldots,i+v\}\not\subset\{s_{\ell+1},\ldots,s_{\rho}\}$.
\end{itemize}
Let 
$P=\{s\in \{s_L,s_{\ell+1},\ldots,s_{\rho}\}\,|\,1\le s\le \ell\}$ and 
$Q=\{i\in\{R,\ell+1,\ldots,\rho\}\,|\,a_i=i\}$.
Notice that $s_L\in P$ and $R\in Q$ 
if and only if $s_L\ne R$.
From the above observations, 
it suffices to consider $\sigma$ satisfying  
$P\subset\{r_1,\ldots,r_d\}$, 
$Q\subset\{r_d+1,\ldots,r_{n}+1\}$,
where we understand $R$ to be $r_n+1$, and    
\[
 \{i\in \{\ell+1,\ldots,\rho\}\,|\,\ell<s_i\le \rho\}
=   
\begin{cases}
 \{\ell+1,\ldots,\rho\}\setminus Q & (s_L=R),\\   
 \{\ell+1,\ldots,\rho\}\setminus (Q\setminus\{R\}) & (s_L\ne R).
\end{cases}
\]
Put $t=|Q|$. 
One sees that 
$|P|=t$, $1\le t\le \min\{d,n-d+1\}$
and, moreover, 
\[
A_{\sigma}B_{\sigma}=(-1)^tA_{P,Q}B_{P,Q},
\]
where 
$A_{P,Q}:=[a'_R,a'_1,\ldots,a'_{\rho}]$ 
and
$B_{P,Q}:=[b'_R,b'_1,\ldots,b'_{\rho}]$ 
are respectively defined by 
\[
a'_i=
\begin{cases}
L & (i=R), \\
i' & (\ell+1\le i\le \rho),\\
Q_{t+1-j} & (\text{$i=P_j$ for some $1\le j\le t$}), \\
i & (\text{otherwise}),
\end{cases}
\quad 
b'_i=
\begin{cases}
i' & (1\le i\le \ell),\\
P_{j} & (\text{$i=Q_{t+1-j}$ for some $1\le j\le t$}), \\
i & (\text{otherwise}).
\end{cases}
\]
Therefore,
writing 
$P=\{P_1=r_{p_1},\ldots,P_t=r_{p_t}\}$
with $1\le p_1<\cdots<p_t\le d$ and 
$Q=\{Q_t=r_{q_t}+1,\ldots,Q_1=r_{q_1}+1\}$
(resp. $Q=\{Q_t=R=r_{q_1}+1,Q_{t-1}=r_{q_t}+1,\ldots,Q_1=r_{q_2}+1\}$) if $s_L=R$ (resp. $s_L\ne R$) with 
$d\le q_t<\cdots<q_1\le n$, 
we have 
\begin{equation}
\label{for:Plucker second step}
\sum_{\substack{\sigma=(s_L,s_{\ell+1},\cdots,s_{\rho}) \\
R\le s_L<s_{\ell+1}<\cdots<s_{\rho}\le \rho \\
\sigma \ne (R,\ell+1,\ldots,{\rho})}}
A_{\sigma}B_{\sigma}
=\sum^{\min\{d,n-d+1\}}_{t=1}(-1)^t
\sum_{\substack{
1\le p_1<\cdots<p_t\le d \\
d\le q_t<\cdots<q_1\le n}}A_{P,Q}B_{P,Q}.    
\end{equation}
Hence, combining \eqref{for:Plucker first step} and  
\eqref{for:Plucker second step}
with 
\begin{align*}
A_{P,Q}
=(-1)^{\varepsilon_{P,Q}}
S^{(r-1)}_{\rem^{p_1,\ldots,p_t}_{q_t,\ldots,q_1}(\lambda)}(X),
\quad 
B_{P,Q}
=(-1)^{\varepsilon_{P,Q}}
S^{(r)}_{\add^{p_1,\ldots,p_t}_{q_t,\ldots,q_1}(\lambda)}(X),
\end{align*}
again derived from \eqref{for:9thJT-nnsy} 
where
\[
\varepsilon_{P,Q}
=
\begin{cases}
\sum^{t}_{i=1}(Q_i-P_{t+1-i}-i) & 
\text{($s_L=R$)},\\
\sum^{t-1}_{i=1}(Q_i-P_{t+1-i}-i)+P_1 & \text{(otherwise)},
\end{cases}
\]
we obtain the desired formula \eqref{for:9thPlucker}.

One can similarly prove \eqref{for:9thDualPlucker}
by using the matrix 
$E_{(\lambda'-(1^{\ell}))'}^{(r)}(X)\,\square\,E_{(\lambda'+(1^{\ell}))'}^{(r+1)}(X)$,
where $E_{\lambda}^{(r)}(X)$ is also defined in Theorem \ref{thm:JT-nnsy}.
\end{proof}


\begin{example}
When $\lambda=(3,2,2,1)$ and $d=2$, we have 
\begin{align*}
\ytableausetup{boxsize=12pt,aligntableaux=center}
\scriptsize
\begin{ytableau}
     0 &      1 &      2 \\
\mb{1} &      0 \\
\mb{2} & \mb{1} \\
\mb{3}   
\end{ytableau}
\cdot
\begin{ytableau}
\mb{1} &      0 & 1\\
\mb{2} & \mb{1} \\
\mb{3} & \mb{2} \\
\mb{4}
\end{ytableau}
\,
&\overset{\eqref{for:9thPlucker}}{=}
\,
\scriptsize
\begin{ytableau}
 0     & 1  \\
\mb{1} \\
\mb{2} \\
\mb{3}
\end{ytableau}
\cdot 
\begin{ytableau}
\mb{1} &      0 &      1 & 2\\
\mb{2} & \mb{1} &      0 \\
\mb{3} & \mb{2} & \mb{1} \\
\mb{4}
\end{ytableau}
\,+\,
\scriptsize
\begin{ytableau}
     0 &      1 & 2 \\
\mb{1} &      0 & 1 \\
\mb{2} & \mb{1} & 0 \\
\mb{3} & \mb{2} & \mb{1}
\end{ytableau}
\cdot 
\begin{ytableau}
\mb{1} \\
\mb{2} \\
\mb{3} \\
\mb{4}
\end{ytableau}
\,+\,
\scriptsize
\begin{ytableau}
     0 &      1 & 2 \\
\mb{1} &      0 & 1 \\
\mb{2} & \mb{1} & 0 \\
\mb{3} & \mb{2} & \mb{1} \\
\mb{4} & \mb{3}
\end{ytableau}
\cdot 
\begin{ytableau}
\mb{1} \\
\mb{2} 
\end{ytableau}
\\
&
\qquad 
\,+\,
\scriptsize
\begin{ytableau}
     0 &      1 & 2 \\
\mb{1} &      0 \\
\mb{2} & \mb{1} \\
\mb{3} & \mb{2} 
\end{ytableau}
\cdot 
\begin{ytableau}
\mb{1} &      0 & 1\\
\mb{2} & \mb{1} \\
\mb{3} \\
\mb{4}
\end{ytableau}
\,+\,
\scriptsize
\begin{ytableau}
 0     & 1      & 2 \\
\mb{1} & 0 \\
\mb{2} & \mb{1} \\
\mb{3} & \mb{2} \\
\mb{4} & \mb{3}
\end{ytableau}
\cdot 
\begin{ytableau}
\mb{1} &      0 & 1 \\
\mb{2} & \mb{1} 
\end{ytableau}
\,-\,
\scriptsize
\begin{ytableau}
     0 &      1 & 2 \\
\mb{1} &      0 & 1 \\
\mb{2} & \mb{1} & 0 \\
\mb{3} & \mb{2} & \mb{1} \\
\mb{4} & \mb{3} & \mb{2}
\end{ytableau}
\cdot 
\begin{ytableau}
\mb{1} 
\end{ytableau}
\,,\\
\scriptsize
\begin{ytableau}
     0 &      1 &      2 \\
\mb{1} &      0 \\
\mb{2} & \mb{1} \\
\mb{3}   
\end{ytableau}
\cdot
\begin{ytableau}
1      & 2 & 3\\
0      & 1 \\
\mb{1} & 0 \\
\mb{2}
\end{ytableau}
\,
&\overset{\eqref{for:9thDualPlucker}}{=}
\,
\scriptsize
\begin{ytableau}
 0     & 1 & 2 \\
\mb{1} & 0 \\
\mb{2} 
\end{ytableau}
\cdot 
\begin{ytableau}
1      &      2 & 3 \\
0      &      1 \\
\mb{1} &      0 \\
\mb{2} & \mb{1} \\
\mb{3}
\end{ytableau}
\,+\,
\scriptsize
\begin{ytableau}
     0 &      1 & 2 \\
\mb{1} &      0 & 1 \\
\mb{2} & \mb{1} & 0 \\
\mb{3} & \mb{2} & \mb{1}
\end{ytableau}
\cdot 
\begin{ytableau}
1 & 2 & 3 \\
0 
\end{ytableau}
\,+\,
\scriptsize
\begin{ytableau}
     0 &      1 & 2 & 3 \\
\mb{1} &      0 & 1 & 2 \\
\mb{2} & \mb{1} & 0 \\
\mb{3} & \mb{2} & \mb{1} 
\end{ytableau}
\cdot 
\begin{ytableau}
1 \\
0 
\end{ytableau}
\\
&
\qquad 
\,+\,
\scriptsize
\begin{ytableau}
     0 &      1 & 2 \\
\mb{1} &      0 & 1 \\
\mb{2} & \mb{1} & 0 \\
\mb{3} 
\end{ytableau}
\cdot 
\begin{ytableau}
     1 & 2 & 3 \\
     0 \\
\mb{1} \\
\mb{2}
\end{ytableau}
\,+\,
\scriptsize
\begin{ytableau}
 0     & 1      & 2 & 3 \\
\mb{1} & 0      & 1 & 2 \\
\mb{2} & \mb{1} & 0 \\
\mb{3} 
\end{ytableau}
\cdot 
\begin{ytableau}
     1 \\
     0 \\
\mb{1} \\
\mb{2} 
\end{ytableau}
\,-\,
\scriptsize
\begin{ytableau}
     0 &      1 &      2 & 3 \\
\mb{1} &      0 &      1 & 2 \\
\mb{2} & \mb{1} &      0 & 1 \\
\mb{3} & \mb{2} & \mb{1} & 0 \\
\end{ytableau}
\,.
\end{align*}
\end{example}

Now, let us again consider the case of a rectangle $\lambda=[\,p\,|\,q\,]$.
Letting $n=1$, $k=1$ and $\ell=\rho=p$ 
in \eqref{for:9thPlucker},
one obtains \eqref{square}.
More generally, we can prove the following result,
which gives \eqref{square} when $a=b=1$.
Here, we employ the notation
$[\,p\,|\,q\,]^{l}_{k}:=((q+1)^{l},q^{p-l},k)$,
$[\,p\,|\,q\,]^{l}:=[\,p\,|\,q\,]^{l}_{0}$ and 
$[\,p\,|\,q\,]_k:=[\,p\,|\,q\,]^{0}_{k}$
defined in \cite{gps06}.



\begin{theorem}
\label{square general}
For $p,q\ge 1$ and $a,b\ge 0$ satisfying  $a\le q$ and $a+b\le p+1$,
we have 
\begin{align}
\label{for:9thRectangle general}
\begin{split}
&
(-1)^{a+b}S_{[\,p+1\,|\,q+b-1\,]}^{(r)}(X)
S_{[\,p-1\,|\,q-a\,]^{p-a-b+1}}^{(r-1)}(X)\\
&=S_{[\,p\,|\,q+b-1\,]}^{(r)}(X)
S_{[\,p\,|\,q-a+1\,]}^{(r-1)}(X)
-S_{[\,p\,|\,q-a\,]}^{(r)}(X)
S_{[\,p\,|\,q+b\,]}^{(r-1)}(X)\\
&\quad +\sum^{a+b-3}_{t=0}(-1)^{t-1}
S_{[\,p\,|\,q+b-1\,]_{q-a+t+1}}^{(r)}(X)
S_{[\,p-1\,|\,q-a\,]^{p-t-1}}^{(r-1)}(X).
\end{split}
\end{align}
\end{theorem}
\begin{proof}
From \eqref{for:Plucker first step}
with $\lambda=[\,p\,|\,q\,]$,
we have 
\begin{align*}
 S_{[\,p\,|\,q-a\,]}^{(r)}(X)
 S_{[\,p\,|\,q+b\,]}^{(r-1)}(X)
&=S_{[\,p\,|\,q-a+1\,]}^{(r-1)}(X)
S_{[\,p\,|\,q+b-1\,]}^{(r)}(X)
+\sum_{1\le t\le p}A_tB_t,
\end{align*}
where 
$A_t:=[R,1,\ldots,\overset{t}{L},\ldots,p]$ and 
$B_t:=[t,1',\ldots,p']$.
It is easy to see that 
\begin{align*}
A_t
&=(-1)^{t-1}S^{(r-1)}_{[\,p-1\,|\,q-a\,]^{t-1}}(X),\\
B_t
&=
\begin{cases}
0 & 1\le t\le p+1-a-b,\\
(-1)^{p}S^{(r)}_{[\,p\,|\,q+b-1\,]_{q-a-t+p+1}}(X) & p+2-a-b\le t\le p. 
\end{cases}
\end{align*}
(Notice that when $1\le t\le p+1-a-b$, 
the $L$-th row and the $s'$-th row coincide for $s=t+a+b-1\le p$.)
Hence, we obtain the desired result.

\end{proof}

\begin{example}
When $a+b=2$ in \eqref{for:9thRectangle general}, we reobtain \eqref{for:9thRectangle}.
When $a+b=3$, we have 
\begin{align*}
&-S_{[\,p+1\,|\,q-a+2\,]}^{(r)}(X)
S_{[\,p-1\,|\,q-a\,]^{p-2}}^{(r-1)}(X)\\
& 
=S_{[\,p\,|\,q-a+2\,]}^{(r)}(X)
S_{[\,p\,|\,q-a+1\,]}^{(r-1)}(X)
-S_{[\,p\,|\,q-a\,]}^{(r)}(X)
S_{[\,p\,|\,q-a+3\,]}^{(r-1)}(X) \\
&\quad 
-S_{[\,p\,|\,q-a+2\,]_{q-a+1}}^{(r)}(X)
S_{[\,p-1\,|\,q-a+1\,]}^{(r-1)}(X).
\end{align*}
For example, when $p=q=3$, $a=1$ and $b=2$,
we have 
\begin{align*}
\ytableausetup{boxsize=12pt,aligntableaux=center}
-\,
\scriptsize
\begin{ytableau}
     0 &      1 &      2 & 3 \\
\mb{1} &      0 &      1 & 2 \\
\mb{2} & \mb{1} &      0 & 1 \\
\mb{3} & \mb{2} & \mb{1} & 0  
\end{ytableau}
\cdot
\begin{ytableau}
\mb{1} &      0 & 1\\
\mb{2} & \mb{1} 
\end{ytableau}
\,
=
\,
\scriptsize
\begin{ytableau}
 0     &      1 &      2 & 3 \\
\mb{1} &      0 &      1 & 2 \\
\mb{2} & \mb{1} &      0 & 1 \\
\end{ytableau}
\cdot 
\begin{ytableau}
\mb{1} &      0 &      1 \\
\mb{2} & \mb{1} &      0 \\
\mb{3} & \mb{2} & \mb{1} \\
\end{ytableau}
\,-\,
\scriptsize
\begin{ytableau}
     0 &      1 \\
\mb{1} &      0 \\
\mb{2} & \mb{1} 
\end{ytableau}
\cdot 
\begin{ytableau}
\mb{1} &      0 &      1 & 2 & 3 \\
\mb{2} & \mb{1} &      0 & 1 & 2 \\
\mb{3} & \mb{2} & \mb{1} & 0 & 1  
\end{ytableau}
\,-\,
\scriptsize
\begin{ytableau}
     0 &      1 &      2 & 3 \\
\mb{1} &      0 &      1 & 2 \\
\mb{2} & \mb{1} &      0 & 1 \\
\mb{3} & \mb{2} & \mb{1} 
\end{ytableau}
\cdot
\begin{ytableau}
\mb{1} &      0 & 1 \\
\mb{2} & \mb{1} & 0 
\end{ytableau}
\,.
\end{align*}
\end{example}

\begin{remark}
As demonstrated in this section, further quadratic identities for the ninth variation of Schur function can be systematically derived via the Jacobi–Trudi formula in combination with the Pl\"{u}cker relations.
For example, one may obtain a generalization of the identity for the Schur function given by Gurevich, Pyatov, and Saponov in \cite[Proposition 3.1]{gps10}.
\end{remark}

\begin{remark}
In \cite{hg}, Hamel and Goulden obtained a determinant formula for the Schur function via the outside decomposition of Young diagrams. This formula generalizes various determinant expressions, such as the Jacobi–Trudi formulas, the dual Jacobi–Trudi formulas, the Giambelli formula, and the Lascoux–Pragacz formula \cite{lp88}.
We expect to obtain similar algebraic relations for the ninth variation of Schur functions by combining the Hamel–Goulden formula and the Pl\"{u}cker relations, following the same strategy as above.
As a further generalization of the ninth variation of the Schur function, Bachmann and Charlton \cite{bc} introduced the tenth variation of Schur function.
We expect to obtain similar results for this version without the “diagonal conditions,” which will be explained in the next section,
by taking special sums, as in \cite{nt}.
\end{remark}


\section{Application: Diagonally constant Schur multiple zeta values}
\label{application}
\ytableausetup{boxsize=normal,aligntableaux=center}

As an application of algebraic relations for $S^{(r)}_{\lambda/\mu}(X)$ or $S^{\FK}_{\lambda/\mu}(W)$ obtained in the previous sections, we derive corresponding relations for a special type of $M$-truncated Schur multiple zeta functions, defined for an index $\bs=(s_{i,j})\in \T(\lambda/\mu,\mathbb{C})$ by 
\[
 \zeta^M_{\lambda/\mu}(\bs)
:=\sum_{(t_{i,j})\in \SSYT_M(\lambda/\mu)}\prod_{(i,j)\in D(\lambda/\mu)}t^{-s_{i,j}}_{i,j},
\]
and the limit (if exists)
\[
\zeta_{\lambda/\mu}(\bs)
:=\lim_{M\to\infty}\zeta^{M}_{\lambda/\mu}(\bs),
\]
which we call the Schur multiple zeta function.
It is shown in \cite[Lemma~2.1]{npy} that 
$\zeta_{\lambda/\mu}(\bs)$ converges absolutely for $\bs\in W_{\lambda/\mu}$, where 
\begin{align}
\label{converge_domain}
 W_{\lambda/\mu}
:=
\left\{(s_{ij})\in \T({\lambda/\mu},\mathbb{C})\left|
\begin{array}{l}
 \text{$\mathrm{Re}(s_{i,j})\ge 1$ for all $(i,j)\in D({\lambda/\mu}) \setminus C({\lambda/\mu})$} \\[3pt]
 \text{$\mathrm{Re}(s_{i,j})>1$ for all $(i,j)\in C({\lambda/\mu})$}
\end{array}
\right.
\right\}. 
\end{align}
The Schur multiple zeta function is a simultaneous generalization of both Euler-Zagier type multiple zeta-star function $\zeta^{\star}(s_1,\ldots,s_d):=\lim_{M\to\infty}\zeta^{\star,M}(s_1,\ldots,s_d)$, and the multiple zeta functions $\zeta(s_1,\ldots,s_d):=\lim_{M\to\infty}\zeta^M(s_1,\ldots,s_d)$, where
\begin{align*}
 \zeta^{\star,M}(s_1,\ldots,s_d)
:=\sum_{1\le m_1\le\cdots\le m_d\le M}\frac{1}{m_1^{s_1}\cdots m_d^{s_d}}, 
\quad 
 \zeta^{M}(s_1,\ldots,s_d)
:=\sum_{1\le m_1<\cdots< m_d\le M}\frac{1}{m_1^{s_1}\cdots m_d^{s_d}},
\end{align*}
in the sense that 
\begin{equation}
\label{zetazetas}
\ytableausetup{boxsize=15pt,aligntableaux=center}
\zeta_{(d)}
\left(\ 
\begin{ytableau}
s_1 & \tcdots & s_d
\end{ytableau}
\ \right)
=\zeta^{\star}(s_1,\ldots,s_d),
\quad 
\zeta_{(1^d)}
\left(\ 
\begin{ytableau}
s_1\\
\tvdots\\
s_d
\end{ytableau}
\ \right)
=\zeta(s_1,\ldots,s_d).
\end{equation}
If the index $\bs$ is ``diagonally constant'', that is, 
$\bs\in \T^{\diag}(\lambda/\mu,\mathbb{C})$, where 
\[
\T^{\diag}(\lambda/\mu,\mathbb{C})
:=\{(t_{i,j})\in \T(\lambda/\mu,\mathbb{C})\,|\,t_{i,j}=t_{k,l} \ {\mathrm{if}} \ c(i,j)=c(k,l)\},
\]
then, one easily sees that 
$\zeta^{M}_{\lambda/\mu}(\bs)$ is realized as a specialization of $S^{\FK}_{\lambda/\mu}(W)$ and hence 
of $S^{(r)}_{\lambda/\mu}(X)$ from Lemma~\ref{lemma:nnsy to fk} as follow
(see also \cite[Lemma~4.2]{npy}).

\begin{lemma}
For $\ba=(a_c)_{c\in\mathbb{Z}}$, 
put $W=\{w_{k,c}\}_{k\in [M], c\in\mathbb{Z}}$ 
with $w_{k,c}=k^{-a_c}$, 
and $\bu=\{u^{(t)}_k\}_{k\in [M],t\in [N-1]}$ 
with $u^{(t)}_k=w_{k,t-r}=k^{-a_{t-r}}$.
Define $\ba|_{\lambda/\mu}:=(a_{c(i,j)})_{(i,j)\in D(\lambda/\mu)}\in \T^{\diag}(\lambda/\mu,\mathbb{C})$.
Then, for $m\in\mathbb{Z}$, we have 
\[
\zeta^{M}_{\lambda/\mu}((\tau^m\ba)|_{\lambda/\mu})
=S^{\FK}_{\lambda/\mu}(\tau^m W)
=S^{(r+m)}_{\lambda/\mu}(U_M(\bu)),
\]
where $\tau^m\ba:=(a_{c+m})_{c\in\mathbb{Z}}$ and $U_M(\bu)$ was defined in Theorem~\ref{thm:nnsy special}.
\end{lemma}

From now on, we write $\zeta^{M}_{\lambda/\mu}(\ba|_{\lambda/\mu})$ simply as $\zeta^{M}_{\lambda/\mu}(\ba)$.
The following results are direct consequences of
Theorem~\ref{thm:Giambelli nnsy}, Theorem~\ref{fk9th} and Theorem~\ref{thm:9thPlucker}.
Notice that Corollary~\ref{cor:Schur skew Giambelli} is a skew generalization of \cite[Theorem~4.5]{npy}.

\begin{cor}
\label{cor:Schur skew Giambelli}
Retaining the notations from  Subsection~\ref{sebsec:Giambelli},
for $\ba=(a_c)_{c\in\mathbb{Z}}\in\mathbb{C}^{\mathbb{Z}}$, we have
\begin{equation}
\label{for:Schur zeta skew Giambelli}
 \zeta^{M}_{\lambda/\mu}(\ba)
=(-1)^q\det 
\left[
\begin{array}{c|c}
\left[\zeta^{M}_{(\alpha_i\,|\,\beta_j)}(\ba)\right]_{\substack{1\le i\le p \\ 1\le j\le p}} & 
\left[\zeta^{\star,M}_{\alpha_i-\gamma_j}(\tau^{\gamma_j+1}\ba)\right]_{\substack{1\le i\le p \\ 1\le j\le q}} \\[10pt]
\hline \\[-10pt]
 \left[\zeta^{M}_{\beta_j-\delta_i}(\tau^{-\delta_i-1}\ba)\right]_{\substack{1\le i\le q \\ 1\le j\le p}} & O_{q}
\end{array}
\right].
\end{equation}
\end{cor}

\begin{cor}
Let $\lambda/\mu$ be a skew partition with 
$\lambda=(\lambda_1,\ldots,\lambda_{k+1})$ and 
$\mu=(\mu_1,\ldots,\mu_{k+1})$,
and $\lambda'=(\lambda'_1,\ldots,\lambda'_{l+1})$ and 
$\mu'=(\mu'_1,\ldots,\mu'_{l+1})$ their conjugates, respectively.
\begin{itemize}
\item[(1)] 
It holds that  
\begin{align*}
& 
\zeta^M_{(\lambda_1,\ldots,\lambda_{k+1})/(\mu_1,\ldots,\mu_{k+1})}(\ba)
\cdot \zeta^M_{(\lambda_2,\ldots,\lambda_{k})/(\mu_2,\ldots,\mu_{k})}(\tau^{-1}\ba)\\
&=\zeta^M_{(\lambda_1,\ldots,\lambda_{k})/(\mu_1,\ldots,\mu_{k})}(\ba)
\cdot \zeta^M_{(\lambda_2,\ldots,\lambda_{k+1})/(\mu_2,\ldots,\mu_{k+1})}(\tau^{-1}\ba)
\\
&\ \ \ 
-\zeta^M_{(\lambda_2-1,\ldots,\lambda_{k+1}-1)/(\mu_1,\ldots,\mu_{k})}(\ba)
\cdot \zeta^M_{(\lambda_1+1,\ldots,\lambda_{k}+1)/(\mu_2,\ldots,\mu_{k+1})}(\tau^{-1}\ba).
\end{align*}
\item[(2)]  
It holds that  
\begin{align*}
& \zeta^M_{(\lambda'_1,\ldots,\lambda'_{l+1})'/(\mu'_1,\ldots,\mu'_{l+1})'}(\ba)
 \cdot \zeta^M_{(\lambda'_2,\ldots,\lambda'_{l})'/(\mu'_2,\ldots,\mu'_{l})'}(\tau\ba)
\\
&=\zeta^M_{(\lambda'_1,\ldots,\lambda'_{l})'/(\mu'_1,\ldots,\mu'_{l})'}(\ba)
 \cdot \zeta^M_{(\lambda'_2,\ldots,\lambda'_{l+1})'/(\mu'_2,\ldots,\mu'_{l+1})'}(\tau\ba)\\
&\ \ \ 
-\zeta^M_{(\lambda'_2-1,\ldots,\lambda'_{l+1}-1)'/(\mu'_1,\ldots,\mu'_{l})'}(\ba)
 \cdot  \zeta^M_{(\lambda'_1+1,\ldots,\lambda'_{l}+1)'/(\mu'_2,\ldots,\mu'_{l+1})'}(\tau\ba).
\end{align*}
\end{itemize}
Here, we understand $\zeta^{M}_{\lambda/\mu}(\ba)=0$
if $\lambda/\mu$ is not a skew partition.
\end{cor}

\begin{cor}
\label{cor:9thPlucker}
Let $\ba=(a_c)_{c\in\mathbb{Z}}$.
Let $\lambda$ be a partition having $n$ corners.
Take $1\le d\le n$. 
\begin{enumerate}
\item  
Denote the $d$-th shortest column height of $\lambda$ as $\ell$.
Then, we have  
\begin{align}
\label{for:9thPluckerZeta}
\begin{split}
&\zeta^{M}_{\lambda}(\ba)\zeta^{M}_{\lambda}(\tau^{-1}\ba)
=\zeta^{M}_{\lambda-(1^{\ell})}(\ba)\zeta^{M}_{\lambda+(1^{\ell})}(\tau^{-1}\ba)\\
&\quad +\sum^{\min\{d,n-d+1\}}_{t=1}(-1)^{t-1}
\sum_{\substack{1\le p_1<\cdots<p_t\le d \\ d\le q_t<\cdots<q_1\le n}}
\zeta^{M}_{\add^{p_1,\ldots,p_t}_{q_1,\ldots,q_t}(\lambda)}(\ba)
\zeta^{M}_{\rem^{p_1,\ldots,p_t}_{q_1,\ldots,q_t}(\lambda)}(\tau^{-1}\ba).
\end{split}
\end{align}
\item 
Denote the $d$-th shortest row length of $\lambda$ as $\ell$.
Then, we have  
\begin{align}
\label{for:9thDualPluckerZeta}
\begin{split}
&\zeta^{M}_{\lambda}(\ba)\zeta^{M}_{\lambda}(\tau\ba)
=\zeta^{M}_{(\lambda'-(1^{\ell}))'}(\ba)\zeta^{M}_{(\lambda'+(1^{\ell}))'}(\tau\ba)\\
&\quad +\sum^{\min\{d,n-d+1\}}_{t=1}(-1)^{t-1}
\sum_{\substack{1\le p_1<\cdots<p_t\le d \\ d\le q_t<\cdots<q_1\le n}}
\zeta^{M}_{(\add^{p_1,\ldots,p_t}_{q_1,\ldots,q_t}(\lambda'))'}(\ba)
\zeta^{M}_{(\rem^{p_1,\ldots,p_t}_{q_1,\ldots,q_t}(\lambda'))'}(\tau\ba).
\end{split}
\end{align}
\end{enumerate}
\end{cor}

As another application or related topic of the results obtained in the previous sections, we next compute certain special values of the diagonally constant Schur multiple zeta values
$\zeta_{\lambda/\mu}(\ba)$ at positive integer points, and, more generally,
the regularized Schur multiple zeta values
$\zeta^{\ast}_{\lambda/\mu}(\ba)$ introduced in \cite{bc}.
This is a generalization of the stuffle (or harmonic) regularized multiple zeta values
$\zeta^{\ast}(\bk)$ introduced in \cite{ikz06}
satisfying that $\zeta^{\ast}_{\lambda/\mu}(\bk)=\zeta_{\lambda/\mu}(\bk)$
whenever the latter converges.
For the precise definitions of $\zeta^{\ast}(\bk)$ and 
$\zeta^{\ast}_{\lambda/\mu}(\bk)$, see the Appendix.

Let $\alpha,\beta,\gamma$ be positive integers.
For $p,q\ge 1$ and $m\in\mathbb{Z}$, define 
\begin{align*}
R^{(m)}_{p,q}=
R^{(m)}_{p,q}(\alpha,\beta,\gamma)
:=\zeta^{\ast}_{[\,p\,|\,q\,]}(\tau^m\ba),    
\end{align*}
where $\ba=(a_c)_{c\in\mathbb{Z}}$ with 
$a_{c}=\gamma$ if $c<0$, $\beta$ if $c=0$ and $\alpha$ otherwise.
Moreover, we put 
$R^{(m)}_{p,q}:=1$ if $p=0$ or $q=0$.
This can be illustrated for $m\ge 0$ by using Young tableau as follows.
\[
\ytableausetup{centertableaux, boxsize=1em}
R^{(m)}_{p,q}
=
\zeta^{\ast}_{[\,p\,|\,q\,]}
\left(\
\begin{ytableau}
\alpha & \tcdots & \tcdots & \tcdots & \tcdots & \tcdots & \tcdots & \alpha \\
\tvdots & & & & & & & \tvdots \\
\alpha & \tcdots & \tcdots & \tcdots & \tcdots & \tcdots & \tcdots & \alpha \\
*(gray)\beta & \alpha & \tcdots & \tcdots & \tcdots & \tcdots & \tcdots & \alpha \\
\gamma & *(gray)\beta & \alpha & \tcdots & \tcdots & \tcdots & \tcdots & \alpha \\
\tvdots & \tddots & *(gray)\tddots & \tddots & & & & \tvdots \\
\gamma & \tcdots & \gamma & *(gray)\beta & \alpha & \tcdots & \tcdots & \alpha \\
\gamma & \tcdots & \tcdots & \gamma & *(gray)\beta & \alpha & \tcdots & \alpha \\
\end{ytableau}
\ \right).
\]
Here, the all-$\beta$ diagonal of the tableau in $R^{(m)}_{p,q}$ starts from the $(m+1)$st row.
As special cases, 
\begin{equation}
\label{for:9thRectangle initial}
R^{(m)}_{p,1}(\alpha,\beta,\gamma)
=\zeta^{\ast}(\{\alpha\}^{m},\beta,\{\gamma\}^{p-m-1}),
\end{equation}
where $\{\alpha\}^m$ means $\alpha$ repeated $m$ times. 
In \cite{by}, Bachmann and the second author considered Schur multiple zeta values filled with alternating entries like a Checkerboard and showed that some Schur multiple zeta values of Checkerboard style, filled with $1$ and $3$, are given by determinants of matrices with odd single zeta values as entries. In analogy with their work, we aim to establish some relations among the values $R_{p,q}^{(m)}$.
We first observe that $R^{(m)}_{p,q}$ satisfies the following relation, which enables us to obtain an explicit expression for it by induction on $q$.

\begin{cor}
\label{square Schur MZV}
For $p,q\ge 1$ and $m\in\mathbb{Z}$, we have 
\begin{align}
\label{for:9thRectangle q}
R^{(m)}_{p,q+1}
R^{(m+1)}_{p,q-1}
&=R^{(m)}_{p,q}
R^{(m+1)}_{p,q}
-
R^{(m)}_{p-1,q}
R^{(m+1)}_{p+1,q}.
\end{align}
\end{cor}
\begin{proof}
This follows directly from \eqref{for:9thRectangle} if all the series appearing above converge.
Otherwise, similarly to the proof of \eqref{for:9thRectangle},
one can prove this from the  dual Jacobi-Trudi formula \eqref{for:regularized dual Jacobi Trudi}
for $\zeta^{\ast}_{\lambda/\mu}(\ba)$
together with Theorem~\ref{djformula}.
\end{proof}

Now, for $a,b,c\ge 0$, 
we concentrate on the shape $[\,a+b+c\,|\,b\,]$:
\begin{align*}
 R^{(a)}_{a+b+c,b}
=R^{(a)}_{a+b+c,b}(\alpha,\beta,\gamma)
=
\zeta^{\ast}_{[\,a+b+c\,|\,b\,]}
\left(\ \ \ \ \ 
\begin{ytableau}
\alpha & \tcdots & \alpha \\
\tvdots & & \vdots \\
\alpha & & \tvdots \\
*(gray)\beta & \tddots & \tvdots \\
\gamma & *(gray)\tddots & \alpha \\
\tvdots & \tddots & *(gray)\beta \\
\tvdots & & \gamma \\
\tvdots & & \tvdots \\
\gamma & \tcdots & \gamma 
\end{ytableau}
\begin{tikzpicture}[overlay, remember picture]
\draw[decorate, decoration={brace, amplitude=6pt}, thick]
($(0,2.1)+(-1.5,-1.3)$)--($(0,0.3) + (-1.5,1.7)$) 
  node[midway,xshift=-12pt] {\small $a$};
\draw[decorate, decoration={brace, amplitude=6pt}, thick]
($(0,0.8)+(-1.5,-1.3)$)--($(0,-1) + (-1.5,1.7)$) 
  node[midway,xshift=-12pt] {\small $b$};
\draw[decorate, decoration={brace, amplitude=6pt}, thick]
($(0,-0.5)+(-1.5,-1.3)$)--($(0,-2.3) + (-1.5,1.7)$) 
  node[midway,xshift=-12pt] {\small $c$};
\draw[decorate, decoration={brace, amplitude=6pt}, thick]
($(0,-0.7)+(-0,-1.3)$)--($(0.3,-0.7) + (-1.6,-1.3)$) 
  node[midway,xshift=0pt,yshift=-12pt] {\small $b$};
\end{tikzpicture}
\ \right).
\ \\[2pt]
\end{align*}
Applying Corollary~\ref{square Schur MZV} inductively with respect to $b$, one can actually reach 
an explicit expression of $R^{(a)}_{a+b+c,b}$.
For example, when $(a,b,c)=(0,3,0)$,
since
\begin{align*}
 R^{(0)}_{3,3}R^{(1)}_{3,1}
=R^{(0)}_{3,2}R^{(1)}_{3,2} 
-R^{(0)}_{2,2}R^{(1)}_{4,2}
&=(R^{(0)}_{3,1}R^{(1)}_{3,1} 
-R^{(0)}_{2,1}R^{(1)}_{4,1})
(R^{(1)}_{3,1}R^{(2)}_{3,1} 
-R^{(1)}_{2,1}R^{(2)}_{4,1})\\
&\qquad
-(R^{(0)}_{2,1}R^{(1)}_{2,1} 
-R^{(0)}_{1,1}R^{(1)}_{3,1})
(R^{(1)}_{4,1}R^{(2)}_{4,1} 
-R^{(1)}_{3,1}R^{(2)}_{5,1}),
\end{align*}
dividing both sides by $R^{(1)}_{3,1}$,
we obtain the expression 
\begin{align}
\label{for:Jacobi Trudi 3}
\begin{split}
R^{(0)}_{3,3}
=\zeta^{\ast}_{[3\,|\,3]}
\left(\,
\begin{ytableau}
\beta & \alpha & \alpha \\
\gamma & \beta & \alpha \\
\gamma & \gamma & \beta
\end{ytableau}
\,\right)
&=R^{(0)}_{3,1}R^{(1)}_{3,1}R^{(2)}_{3,1}
+R^{(0)}_{2,1}R^{(1)}_{2,1}R^{(2)}_{5,1}
+R^{(0)}_{1,1}R^{(1)}_{4,1}R^{(2)}_{4,1}\\
&\qquad 
-R^{(0)}_{3,1}R^{(1)}_{2,1}R^{(2)}_{4,1}
-R^{(0)}_{2,1}R^{(1)}_{4,1}R^{(2)}_{3,1}
-R^{(0)}_{1,1}R^{(1)}_{3,1}R^{(2)}_{5,1}.
\end{split}
\end{align}
When $(\alpha,\beta,\gamma)=(2,3,2)$ or $(1,2,1)$, 
the following results on the initial values  
$R^{(a)}_{a+1+c,1}=\zeta^{\ast}(\{\alpha\}^a,\beta,\{\gamma\}^c)$ 
are obtained in \cite[Theorem~1]{z12} and in the appendix of the present paper, respectively:
\begin{align}
\label{for:Zagier}
\zeta(\{2\}^a,3,\{2\}^c)
&=2\sum^{a+c+1}_{r=1}(-1)^r
\left[\binom{2r}{2a+2}-\left(1-\frac{1}{2^{2r}}\right)
\binom{2r}{2c+1}\right]
\eta(a+c-r+1)\zeta(2r+1),\\
\label{for:regularized 121}
\zeta^{\ast}(\{1\}^a,2,\{1\}^c)
&=(-1)^c\sum^{c}_{s=0}\binom{a+c-s+1}{c-s}
\zeta(a+c-s+2)C_s,
\end{align}
where $\eta(k):=\zeta(\{2\}^k)$ for $k\ge 0$
and $\{C_s\}_{s\ge 0}$ is defined by $C_0:=1$, $C_1:=0$ and for $s\ge 2$
\[ 
C_s:=
\sum_{\substack{k_2,k_3,\ldots,k_s\ge 0\\ 2k_2+3k_3+\cdots+sk_s=s}}\frac{(-1)^{k_2+k_3+\cdots+k_s}}{k_2!k_3!\cdots k_s!}
\left(\frac{\zeta(2)}{2}\right)^{k_2}
\left(\frac{\zeta(3)}{3}\right)^{k_3}\cdots
\left(\frac{\zeta(s)}{s}\right)^{k_s}.
\]
Substituting these into the above formula, 
we have, respectively,
\begin{align*}
R^{(0)}_{3,3}(2,3,2)
&= -\frac{801675}{1024} \zeta(3) \zeta(7) \zeta(11)
 -\frac{1058211}{512} \zeta(5)^2 \zeta(11)
 +\frac{160335}{64} \eta(1) \zeta(3) \zeta(5) \zeta(11)\\
&\qquad -\frac{32067}{64} \eta(1)^2 \zeta(3)^2 \zeta(11)
 -\frac{404495}{256}  \zeta(3) \zeta(9)^2
 +\frac{1101387}{256}  \zeta(5) \zeta(7) \zeta(9)\\
&\qquad +\frac{483}{4} \eta(1)  \zeta(3) \zeta(7) \zeta(9)
-\frac{21315}{16} \eta(1)  \zeta(5)^2 \zeta(9)
 -\frac{777}{8} \eta(2)  \zeta(3) \zeta(5) \zeta(9)\\
&\qquad -\frac{1491}{16} \eta(1)^2  \zeta(3) \zeta(5) \zeta(9)
 -\frac{651}{4} \eta(1) \eta(2)  \zeta(3)^2 \zeta(9)
 +\frac{2667}{8} \eta(1)^3 \zeta(3)^2 \zeta(9)\\
&\qquad -\frac{3426525}{2048}  \zeta(7)^3
 +\frac{54873}{128} \eta(1)  \zeta(5) \zeta(7)^2
 -\frac{1575}{2} \eta(2)  \zeta(3) \zeta(7)^2\\
&\qquad +\frac{128331}{128} \eta(1)^2 \zeta(3) \zeta(7)^2
 +\frac{6705}{16} \eta(2)  \zeta(5)^2 \zeta(7)
 -\frac{6219}{16} \eta(1)^2  \zeta(5)^2 \zeta(7)\\
&\qquad +\frac{6849}{8} \eta(1) \eta(2)  \zeta(3) \zeta(5) \zeta(7)
 -\frac{17163}{16} \eta(1)^3 \zeta(3) \zeta(5) \zeta(7)
 -\frac{225}{4} \eta(2)^2  \zeta(3)^2 \zeta(7)\\
&\qquad +\frac{225}{4} \eta(1)^2 \eta(2) \zeta(3)^2 \zeta(7)
-\frac{855}{2} \eta(1) \eta(2)  \zeta(5)^3
 +495 \eta(1)^3 \zeta(5)^3\\
&\qquad +\frac{81}{2} \eta(2)^2 \zeta(3) \zeta(5)^2
-\frac{81}{2} \eta(1)^2 \eta(2) \zeta(3) \zeta(5)^2,\\
R^{(0)}_{3,3}(1,2,1)
&= -9\zeta(4)^3 -24\zeta(2)\zeta(5)^2 -20\zeta(3)^2\zeta(6)  +30\zeta(2)\zeta(4)\zeta(6)
+24\zeta(3)\zeta(4)\zeta(5). 
\end{align*}
When $(\alpha,\beta,\gamma)=(1,2,1)$, 
we obtain the following more general expression.
\begin{prop}
It holds that 
\begin{align*}
R^{(a)}_{a+b+c,b}(1,2,1)
&=(-1)^{bc+\frac{1}{2}b(b-1)}
\sum^{c+b-1}_{t_1=0}\sum^{c+b-2}_{t_2=0}\cdots \sum^{c}_{t_b=0}C_{t_1}C_{t_2}\cdots C_{t_b}\\
&\quad \times 
\det\left[\binom{a+b+c-i+j-t_i}{b+c-i-t_i}\zeta(a+b+c-i+j-t_i+1)\right]_{1\le i,j\le b}.
\end{align*}
\end{prop}
\begin{proof}
This is obtained by substituting \eqref{for:regularized 121} into the identity  
\begin{align}
\label{for:dual JT for rectangle shape}
 R^{(a)}_{a+b+c,b}
=\det\left[\zeta^{\ast}(\{\alpha\}^{a+j-1},\beta,\{\gamma\}^{c+b-i})\right]_{1\le i,j\le b},
\end{align}
which is a special case of 
\eqref{for:regularized dual Jacobi Trudi} 
and is a generalization of \eqref{for:Jacobi Trudi 3}.
\end{proof}

\begin{example}
When $b=2$ with $c=0,1$, 
we have, respectively,  
\begin{align*}
R_{a+2,2}^{(a)}
&=-(a+2)\zeta(a+3)^2+(a+3)\zeta(a+2)\zeta(a+4),\\
R^{(a)}_{a+3,2}
&=-(a+3)\binom{a+3}{2}\zeta(a+4)^2
+(a+2)\binom{a+4}{2}\zeta(a+3)\zeta(a+5)\\
&\qquad +\frac{1}{2}(a+3)\zeta(2)\zeta(a+2)\zeta(a+4)
-\frac{1}{2}(a+2)\zeta(2)\zeta(a+3)^2.
\end{align*}
Furthermore, 
when $b=3$ with $c=0$, we have 
\begin{align*}
R_{a+3,3}^{(a)}
&=-(a+3)\binom{a+3}{2}\zeta(a+4)^3
-(a+4)\binom{a+4}{2}\zeta(a+2)\zeta(a+5)^2\\
&\qquad -(a+2)\binom{a+5}{2}\zeta(a+3)^2\zeta(a+6)\\
&\qquad +3\binom{a+5}{3}\zeta(a+2)\zeta(a+4)\zeta(a+6)
+6\binom{a+4}{3}\zeta(a+3)\zeta(a+4)\zeta(a+5).
\end{align*}
\end{example}

For general $\alpha,\beta,\gamma$,
it is in general difficult to obtain an explicit expression for $R^{(a)}_{a+b+c,b}$.
As a possible approach, we consider generating functions. 
Let
\[
 F(x,z)=F(\alpha,\beta,\gamma;x,z)
:=\sum_{a,c\ge 0}\zeta^{\ast}(\{\alpha\}^a,\beta,\{\gamma\}^c)x^az^c.
\]
For example, from \cite[Proposition 1]{z12} and 
Theorem~\ref{thm:zast}, respectively, we have
\begin{align*}
-x^2zF(2,3,2;-x^2,-z^2)
&=\frac{\sin\pi z}{\pi}
{}_3F_{2}^{\prime}
\left(\left.\begin{array}{c}
x,-x, 0\\
1+z, 1-z
\end{array} \right\rvert\, 1\right),\\
xF(1,2,1;x,z)
&=-\frac{\psi(z-x+1)-\psi(z+1)}{\Gamma(z+1)e^{\gamma z}}.
\end{align*}
Here, the second factor of the right-hand side of the first equation is the $y$-derivative at $y=0$  of the generalized hypergeometric function 
\begin{align*}
{}_3F_{2}
\left(\left.\begin{array}{c}
x,-x, y\\
1+z, 1-z
\end{array} \right\rvert\, 1\right)
=
\sum_{m=0}^{\infty} \frac{(x)_{m}(-x)_{m}(y)_{m}}{(1+z)_{m}(1-z)_{m}} \frac{1}{m!},
\end{align*}
where $(a)_{m}:=a(a+1) \cdots(a+m-1)$ denotes the Pochhammer symbol,
and $\Gamma(z)$ and $\psi(z):=\frac{d}{dz}\log\Gamma(z)$ are the gamma and the digamma function, respectively. 
For $b\ge 2$, we consider the generating function 
of $R^{(a)}_{a+b+c,b}$ of the form 
\begin{align*}
\Phi_b(x,z)=\Phi_b(\alpha,\beta,\gamma;x,z)
:=x^{\frac{1}{2}b(b-1)}z^{\frac{1}{2}(b-1)(b-2)}
\sum_{a,c\ge 0}R^{(a)}_{a+b+c,b}x^{(b-1)a}z^{(b-1)c}.
\end{align*}
The rest of this section is devoted to expressing $\Phi_b(x,z)$ in terms of $F(x,z)$.

\begin{lemma}
\label{lem:hyper_c}
Let $b\in\mathbb{Z}_{\ge 2}$ and $k_1,\ldots,k_b,\ell_1,\ldots,\ell_b\in\mathbb{Z}_{\ge 0}$.
Put $k=k_1+\cdots+k_b$ and $\ell=\ell_1+\cdots+\ell_b$.
Assume that there exist $i,j\in [b]$ such that $k_i=\ell_j=0$.
Then, we have 
\begin{equation}
\label{for:generating function product}
\begin{split}
&\sum_{a,c\ge 0}
\prod^{b}_{i=1}\zeta^{\ast}(\{\alpha\}^{a+k_i},\beta,\{\gamma\}^{c+\ell_i})\cdot x^{(b-1)a+k-k_1}
z^{(b-1)c+\ell-\ell_1}\\
&=\frac{1}{(2\pi i)^{2(b-1)}}\int_{\mathcal{C}}
F(x_2\cdots x_b,z_2\cdots z_b)\prod^{b}_{i=2}
x_i^{k_i-k_1-1}
z_i^{\ell_i-\ell_1-1}
F\left(\frac{x}{x_i},\frac{z}{z_i}\right)dx_idz_i,
\end{split}
\end{equation}
where $\mathcal{C}$ is the positively oriented product contour
\[
\mathcal{C} 
:=\{ |x_2| = 1 \} \times \cdots 
\times \{ |x_b| = 1 \} 
\times \{ |z_2| = 1 \} \times \cdots 
\times \{ |z_b| = 1 \}.
\]
\end{lemma}
\begin{proof}
For simplicity, put 
$Z(a,c)=\zeta^{\ast}(\{\alpha\}^{a},\beta,\{\gamma\}^{c})$
and understand that $Z(a,c)=0$ whenever $a<0$ or $c<0$.
We see that the right-hand side of \eqref{for:generating function product} equals 
\begin{align*}
&\frac{1}{(2\pi i)^{2(b-1)}}
\int_{\mathcal{C}}
\sum_{a,c\ge 0}\sum_{\substack{a_i,c_i\ge 0 \\ 2\le i\le b}}Z(a,b)x^{\sum^b_{i=2}a_i}z^{\sum^{b}_{i=2}c_i}
\prod^{b}_{i=2}Z(a_i,c_i)
x_i^{a+k_i-k_1-1-a_i}
z_i^{c+\ell_i-\ell_1-1-c_i}
dx_idz_i\\
&=\sum_{a,c\ge 0}
Z(a,b)\prod^{b}_{i=2}Z(a+k_i-k_1,c+\ell_i-\ell_1)
\cdot x^{\sum^{b}_{i=2}(a+k_i-k_1)}
y^{\sum^{b}_{i=2}(c+\ell_i-\ell_1)}\\
&=\sum_{a\ge -k_1,c\ge -\ell_1}
Z(a+k_1,c+\ell_1)\prod^{b}_{i=2}Z(a+k_i,c+\ell_i)
\cdot x^{\sum^{b}_{i=2}(a+k_i)}
z^{\sum^{b}_{i=2}(c+\ell_i)}\\
&=\sum_{a,c\ge 0}\prod^{b}_{i=1}Z(a+k_i,c+\ell_i)
\cdot x^{(b-1)a+k-k_1}
z^{(b-1)c+\ell-\ell_1}.
\end{align*}
In the last equality, we have used the assumption.
\end{proof}

\begin{theorem}
For $b\ge 2$, we have
\begin{align}
\label{for:generating function of R}
 \Phi_b(x,z)
=\frac{1}{(2\pi i)^{2(b-1)}}\int_{\mathcal{C}}
\prod_{2\le i<j\le b}(x_j-x_i)\cdot \frac{F(X,Z)}{X^b}
\prod^{b}_{i=2}\frac{Xx_i-x}{z^i_i}
F\left(\frac{x}{x_i},\frac{z}{z_i}\right)dx_idz_i,
\end{align}
where $X=x_2\cdots x_b$ and $Z=z_2\cdots z_b$.
\end{theorem}
\begin{proof}
Put $\varphi_b(x,z):=x^{-\frac{1}{2}b(b-1)}z^{-\frac{1}{2}(b-1)(b-2)}\Phi_b(x,z)$.
Then, from \eqref{for:dual JT for rectangle shape}, we have 
\begin{align*}
\varphi_b(x,z)
&=\sum_{a,c\ge 0}R^{(a)}_{a+b+c,b}x^{(b-1)a}z^{(b-1)c}\\
&=\sum_{\sigma\in S_b}\mathrm{sgn}(\sigma)\sum_{a,c\ge 0}\left(\prod^{b}_{i=1}\zeta^{\ast}(\{\alpha\}^{a+\sigma(i)-1},\beta,\{\gamma\}^{c+b-i})\right)x^{(b-1)a}z^{(b-1)c},
\end{align*}
where $S_b$ denotes the symmetric group of degree $b$, and $\mathrm{sgn}(\sigma)$ is the signature of $\sigma \in S_b$. 
From \eqref{for:generating function product} with $k_i=\sigma(i)-1$ and $\ell_i=b-i$, which implies  
$k=\ell=\frac{1}{2}b(b-1)$, we see that
\begin{align*}
&\sum_{a,c\ge 0}\left(\prod^{b}_{i=1}\zeta^{\ast}(\{\alpha\}^{a+\sigma(i)-1},\beta,\{\gamma\}^{c+b-i})\right)x^{(b-1)a}z^{(b-1)c}\\
&=\frac{1}{(2\pi i)^{2(b-1)}}\int_{\mathcal{C}}
\frac{F(X,Z)}{x^{\frac{1}{2}b(b-1)-(\sigma(1)-1)}z^{\frac{1}{2}b(b-1)-(b-1)}}\prod^{b}_{i=2}
x_i^{\sigma(i)-1-\sigma(1)}
z_i^{-i}
F\left(\frac{x}{x_i},\frac{z}{z_i}\right)dx_idz_i\\
&=x^{-\frac{1}{2}b(b-1)}z^{-\frac{1}{2}(b-1)(b-2)}
\frac{1}{(2\pi i)^{2(b-1)}}\int_{\mathcal{C}}
\frac{F(X,Z)}{X}\left(\frac{x}{X}\right)^{\sigma(1)-1}\prod^{b}_{i=2}
x_i^{\sigma(i)-1}
\prod^{b}_{i=2}\frac{1}{z^i_i}
F\left(\frac{x}{x_i},\frac{z}{z_i}\right)dx_idz_i.
\end{align*}
Therefore, by applying the Vandermonde determinant, we obtain
\begin{align*}
&\Phi_b(x,z)\\
&=\frac{1}{(2\pi i)^{2(b-1)}}\int_{\mathcal{C}}
\frac{F(X,Z)}{X}
\left(\sum_{\sigma\in S_b}\mathrm{sgn}(\sigma)\left(\frac{x}{X}\right)^{\sigma(1)-1}\prod^{b}_{i=2}
x_i^{\sigma(i)-1}\right)
\prod^{b}_{i=2}\frac{1}{z^i_i}
F\left(\frac{x}{x_i},\frac{z}{z_i}\right)dx_idz_i\\
&=\frac{1}{(2\pi i)^{2(b-1)}}\int_{\mathcal{C}}
\frac{F(X,Z)}{X}\prod_{2\le i<j\le b}(x_j-x_i)\prod^{b}_{i=2}\left(x_i-\frac{x}{X}\right)
\prod^{b}_{i=2}\frac{1}{z^i_i}
F\left(\frac{x}{x_i},\frac{z}{z_i}\right)dx_idz_i\\
&=\frac{1}{(2\pi i)^{2(b-1)}}\int_{\mathcal{C}}
\prod_{2\le i<j\le b}(x_j-x_i)\cdot 
\frac{F(X,Z)}{X^b}
\prod^{b}_{i=2}\frac{Xx_i-x}{z^i_i}
F\left(\frac{x}{x_i},\frac{z}{z_i}\right)dx_idz_i.
\end{align*}
This completes the proof.
\end{proof}

\begin{remark}
Although all the results above concern vertical rectangles, we can naturally consider horizontal rectangles as well.
For instance, for $c\ge 0$, we have 
\begin{align*}
R_{2,c+2}^{(0)}(2,2,1)
&=\zeta_{[\,2\,|\,c+2\,]}
\left(\,
\begin{ytableau}
2 & 2 & 2 & \tcdots & 2\\
1 & 2 & 2 & \tcdots & 2
\end{ytableau}
\begin{tikzpicture}[overlay, remember picture]
\draw[decorate, decoration={brace, amplitude=6pt}, thick]
($(0,0.8)+(-0,-1.3)$)--($(0.3,0.8) + (-1.6,-1.3)$) 
  node[midway,xshift=0pt,yshift=-12pt] {\small $c$};
\end{tikzpicture}
\,\right)\\[20pt]
&=4(1-2^{-2c-3})\zeta(2c+3)\zeta(2c+4)
-4(1-2^{-2c-1})\zeta(2c+2)\zeta(2c+5),\\
R_{2,c+3}^{(-1)}(2,2,1)
&=\zeta_{[\,2\,|\,c+3\,]}
\left(\,
\begin{ytableau}
1& 2 & 2 & 2 & \tcdots & 2\\
1& 1 & 2 & 2 & \tcdots & 2
\end{ytableau}
\begin{tikzpicture}[overlay, remember picture]
\draw[decorate, decoration={brace, amplitude=6pt}, thick]
($(0,0.8)+(-0,-1.3)$)--($(0.3,0.8) + (-1.6,-1.3)$) 
  node[midway,xshift=0pt,yshift=-12pt] {\small $c$};
\end{tikzpicture}
\,\right)\\[20pt]
&=8\left(\zeta^{\star}(1,2c+3)\zeta(2c+5)-\zeta^{\star}(1,2c+5)\zeta(2c+3)\right)\\
&\qquad -4\left(\zeta(2c+4)\zeta(2c+5)-\zeta(2c+3)\zeta(2c+6)\right).
\end{align*}
Actually, from \eqref{for:9thRectangle q}, 
it holds that 
\begin{align*}
R^{(-a)}_{2,a+2+c}
&=R^{(-a-1)}_{1,(a+1)+1+c}R^{(-a)}_{1,a+1+(c+1)}
-R^{(-a-1)}_{1,(a+1)+1+(c+1)}R^{(-a)}_{1,a+1+c}\\
&=\zeta^{\star}(\{1\}^{a+1},\{2\}^{c+1})
\zeta^{\star}(\{1\}^a,\{2\}^{c+2})
-\zeta^{\star}(\{1\}^{a+1},\{2\}^{c+2})
\zeta^{\star}(\{1\}^{a},\{2\}^{c+1}).
\end{align*}
Now the desired formulas follow from this identity 
with $a=0,1$ and the facts
$\zeta^{\star}(\{2\}^{c})=2(1-2^{-2c+1})\zeta(2c)$
and 
\[
\zeta^{\star}(1,\{2\}^{c})
=2\zeta(2c+1),\quad 
\zeta^{\star}(\{1\}^2,\{2\}^c)
=4\zeta^{\star}(1,2c+1)-2\zeta(2c+2),
\]
which are obtained in \cite{zl05} and \cite[Lemma~5]{oz08}, respectively.
We remark that a general formula for 
$\zeta^{\star}(\{1\}^{a},\{2\}^{c})$
is obtained in \cite[Theorem~1.3]{ce23}.




\end{remark}

\appendix 

\section{}

\subsection{Regularized Schur multiple zeta values}

We briefly review the regularization of multiple zeta values, following \cite{ikz06}. 
Let $\mathfrak{H}:=\mathbb{Q}\langle e_0,e_1\rangle$ be the
noncommutative polynomial algebra in $e_0,e_1$ over $\mathbb{Q}$, called the Hoffman algebra, and 
$\mathfrak{H}^1:=\mathbb{Q}+ e_1\mathfrak{H}$ and $\mathfrak{H}^0:=\mathbb{Q}+ e_1\mathfrak{H}e_0$. Let $Z:\mathfrak{H}^0\to\mathbb{R}$ be the evaluation map defined by $Z(z_{k_1}\cdots z_{k_d}):=\zeta(k_1,\ldots,k_d)$,
where $z_k:=e_1e^{k-1}_0$ for $k\in\mathbb{N}$.
Let $\ast$ and $\shuffle$ be the stuffle (or harmonic) and shuffle product on $\mathfrak{H}^1$, respectively, which make $\mathfrak{H}^1_{\ast}:=(\mathfrak{H}^1,\ast)$ and $\mathfrak{H}^1_{\shuffle}:=(\mathfrak{H}^1,\shuffle)$ commutative algebras. We know that there are unique $\mathbb{Q}$-algebra homomorphisms $Z^{\ast}:\mathfrak{H}^1_{\ast}\to\mathbb{R}[T]$ and $Z^{\shuffle}:\mathfrak{H}^1_{\shuffle}\to\mathbb{R}[T]$ satisfying  
$Z^{\ast}|_{\mathfrak{H}^0}=Z^{\shuffle}|_{\mathfrak{H}^0}=Z$ and $Z^{\ast}(z_1)=Z^{\shuffle}(z_1)=T$. 
Explicitly, if $w\in \mathfrak{H}^{1}$ is expressed as 
\[
w
=\sum^m_{i=0}u_i\ast z^{\ast \,i}_1
=\sum^n_{j=0}v_i\shuffle z^{\shuffle \,i}_1
\]
for some $u_i,v_i\in \mathfrak{H}^{0}$, 
where $z^{\ast \,i}_1=\underbrace{z_1 \ast \cdots \ast  z_1}_{i}$ and $z^{\shuffle \,i}_1=\underbrace{z_1 \shuffle \cdots \shuffle z_1}_{i}$), then we have 
\begin{equation}
\label{for:regularizations}
Z^{\ast}(w)=\sum^m_{i=0}Z(u_i)T^i,\quad 
Z^{\shuffle}(w)=\sum^n_{i=0}Z(v_i)T^i.
\end{equation}
Now, for an index $\bk=(k_1,\ldots,k_d)\in \mathbb{N}^d$, 
the stuffle (or harmonic) and shuffle regularized multiple zeta values are respectively defined by 
$Z^{\ast}(\bk;T):=Z^{\ast}(z_{k_1}\cdots z_{k_d})$ and $Z^{\shuffle}(\bk;T):=Z^{\shuffle}(z_{k_1}\cdots z_{k_r})$. 
In particular, we put 
$\zeta^{\ast}(\bk):=Z^{\ast}(\bk;0)$ and 
$\zeta^{\shuffle}(\bk):=Z^{\shuffle}(\bk;0)$.
Remark that $Z^{\ast}(\bk;T)$ can be characterized as the asymptotic property
\[
 \zeta^M(\bk)
=Z^{\ast}(\bk;\log M+\gamma)+O\left(\frac{\log^J M}{M}\right)
\ \ \text{for some $J>0$ as $M\to\infty$},
\]
where $\gamma$ is the Euler constant.
As an extension of this, 
Bachmann and Charlton \cite[Lemma~3.1]{bc} showed that 
for any $\bk\in \T(\lambda/\mu,\mathbb{N})$  
there exists a unique polynomial
$Z^{\ast}_{\lambda/\mu}(\bk;T)\in\mathbb{R}[T]$ satisfying  
\[
 \zeta^{M}_{\lambda/\mu}(\bk)
=Z^{\ast}_{\lambda/\mu}(\bk;\log M+\gamma)+O\left(\frac{\log^J M}{M}\right)
\ \ \text{for some $J>0$ as $M\to\infty$}.
\]
We call $Z^{\ast}_{\lambda/\mu}(\bk;T)$ the regularized Schur multiple zeta values.
Note that $Z^{\ast}_{\lambda/\mu}(\bk;T)=\zeta_{\lambda/\mu}(\bk)$ if $\bk\in W_{\lambda/\mu}$, that is, when $\zeta_{\lambda/\mu}(\bk)$ converges. 
Moreover, by definition, 
\[
\ytableausetup{boxsize=15pt,aligntableaux=center}
Z^{\ast}_{(1^d)}
\left(\ 
\begin{ytableau}
k_1\\
\tvdots\\
k_d
\end{ytableau}
\ ;T \right)
=Z^{\ast}(k_1,\ldots,k_d;T).
\]
Similarly to the above, we define $\zeta^{\ast}_{\lambda/\mu}(\bk):=Z^{\ast}_{\lambda/\mu}(\bk;0)$,
which was one of our target in the previous section.
When $\bk\in \T^{\diag}(\lambda/\mu,\mathbb{N})$,
that is, $\bk=\ba|_{\lambda/\mu}$ for some 
$\ba=(a_c)_{c\in\mathbb{Z}}$ with $a_c\in\mathbb{N}$,
$Z^{\ast}_{\lambda/\mu}(\bk;T)$
can be calculated 
by using the dual Jacobi–Trudi formula as follows.
Here, we also write $Z^{\ast}_{\lambda/\mu}(\ba|_{\lambda/\mu};T)$ simply as 
$Z^{\ast}_{\lambda/\mu}(\ba;T)$.

\begin{lemma}[{A special case of \cite[Theorem~3.3]{bc}}]
\label{lem:regularized dual Jacobi Trudi}
For $\ba=(a_c)_{c\in\mathbb{Z}}$ with $a_c\in\mathbb{N}$,
we have 
\begin{align*}
Z^{\ast}_{\lambda/\mu}(\ba;T)
=\det \left[Z^{\ast}_{\lambda'_i-\mu'_j-i+j}(\tau^{-\mu'_j+j-1}\ba\,;T)\right]_{1\le i,j\le \ell(\lambda')}.
\end{align*}
In particular,
\begin{align}
\label{for:regularized dual Jacobi Trudi}
\zeta^{\ast}_{\lambda/\mu}(\ba)
=\det \left[\zeta^{\ast}_{\lambda'_i-\mu'_j-i+j}(\tau^{-\mu'_j+j-1}\ba)\right]_{1\le i,j\le \ell(\lambda')}.
\end{align}
\end{lemma}

\subsection{Explicit evaluations for $\zeta^{\shuffle}(\{1\}^a,2,\{1\}^c)$ and $\zeta^{\ast}(\{1\}^a,2,\{1\}^c)$}

The rest of this appendix is devoted to proving the following theorem.

\begin{theorem}
\label{thm:explicitregularizations}
For $a,c\in\mathbb{Z}_{\ge 0}$,
$\zeta^{\shuffle}(\{1\}^a,2,\{1\}^c),\zeta^{\ast}(\{1\}^a,2,\{1\}^c)\in\mathbb{Q}[\pi^2,\zeta(3),\zeta(5),\zeta(7),\ldots]$.
More explicitly, we have 
\begin{align}
\label{for:regzeta shuffle}
\zeta^{\shuffle}(\{1\}^a,2,\{1\}^c)
&=(-1)^c\binom{a+c+1}{b}\zeta(a+c+2),\\
\label{for:regzeta stuffle}
\zeta^{\ast}(\{1\}^a,2,\{1\}^c)
&=(-1)^b\sum^{c}_{s=0}\binom{a+b-s+1}{c-s}
\zeta(a+c-s+2)C_s,
\end{align}
where $\{C_s\}_{s\ge 0}$ is defined in \eqref{for:regularized 121}.
\end{theorem}

To prove this theorem,
we first derive the following identities involving shuffle products.

\begin{lemma}
For $a,c\in\mathbb{Z}_{\ge 0}$, we have 
\begin{align}
\label{for:z1nz2shz1m}
 z^a_1z_2\shuffle z^c_1
&=\sum^{c}_{k=0}\binom{a+c-k+1}{c-k}z^{a+c-k}_1z_2z^{k}_1,\\
\label{for:z1nz2z1m}
 z^a_1z_2z^c_1
&=\sum^{c}_{k=0}\frac{(-1)^{c-k}}{k!}\binom{a+c-k+1}{c-k}z^{a+c-k}_1z_2\shuffle z^{\shuffle \,k}_1.
\end{align}
\end{lemma}
\begin{proof}
Since 
$z^a_1z_2\shuffle z^c_1
=z_1(z^{a-1}_1z_2\shuffle z^c_1)+z_1(z^a_1z_2\shuffle z^{c-1}_1)$,
one easily shows \eqref{for:z1nz2shz1m} by induction on $a+c$.
We can also prove \eqref{for:z1nz2z1m} by induction on $c$.
Actually, the case $c=0$ is clear. 
Now assume that it holds for $<c$. 
Then, from \eqref{for:z1nz2shz1m}, we have  
\begin{align*}
z^{a}_1z_2z^{c}_1
&=\frac{1}{c!}z^a_1z_2\shuffle z^{\shuffle \,c}_1
-\sum^{c-1}_{k=0}\binom{a+c-k+1}{c-k}z^{a+c-k}_1z_2z^{k}_1\\
&=\frac{1}{c!}z^a_1z_2\shuffle z^{\shuffle \,c}_1-\sum^{c-1}_{k=0}\binom{a+c-k+1}{c-k}\sum^{k}_{l=0}\frac{(-1)^{k-l}}{l!}\binom{a+c-l+1}{k-l}z^{a+c-l}_1z_2\shuffle z^{\shuffle \,l}_1\\
&=\frac{1}{c!}z^a_1z_2\shuffle z^{\shuffle \,c}_1\\
&\quad 
-\sum^{c-1}_{l=0}\frac{(-1)^{c-l}}{l!}
\left\{
\sum^{c-1}_{k=l}(-1)^{c-k}
\binom{a+c-k+1}{c-k}\binom{a+c-l+1}{k-l}\right\}
z^{a+c-l}_1z_2\shuffle z^{\shuffle \,l}_1.
\end{align*}
Since the inner sum on the rightmost side equals $-\binom{a+c-l+1}{c-l}$,
the proof is complete.
\end{proof}

\begin{proof}[Proof of Theorem~\ref{thm:explicitregularizations}]
From \eqref{for:z1nz2z1m} and \eqref{for:regularizations}, we have 
\begin{align}
Z^{\shuffle}(\{1\}^a,2,\{1\}^c;T)
&=\sum^{c}_{k=0}\frac{(-1)^{c-k}}{k!}\binom{a+c-k+1}{c-k}Z(z^{a+c-k}_1z_2)T^k\nonumber \\
\label{for:Zsh}
&=\sum^{c}_{k=0}\frac{(-1)^{c-k}}{k!}\binom{a+c-k+1}{c-k}\zeta(a+c+2-k)T^k.
\end{align}
Here, we have used the duality formula 
$Z(z^{l}_1z_2)=\zeta(\{1\}^{l},2)=\zeta(l+2)$
with $l\ge 1$ for multiple zeta values.
Now \eqref{for:regzeta shuffle} is obtained by letting $T=0$. 

To prove \eqref{for:regzeta stuffle}, 
we employ the theory of regularization.
Define the $\mathbb{R}$-linear map $\rho:\mathbb{R}[T]\to\mathbb{R}[T]$ by 
$\rho(e^{Tz}):=A(z)e^{Tz}$, where 
\[
A(z):=\Gamma(z+1)e^{\gamma z}
=\exp\left(\sum^{\infty}_{l=2}\frac{(-1)^l}{l}\zeta(l)z^l\right).
\]
Note that $\rho^{-1}(e^{Tz})=A(z)^{-1}e^{Tz}$.
The fundamental theorem of the regularization of multiple zeta values \cite[Theorem~1]{ikz06} asserts that 
\[
Z^{\ast}(\bk;T)
=\rho^{-1}\left(Z^{\shuffle}(\bk;T)\right).
\]
From this together with \eqref{for:Zsh}, 
we have 
\begin{align*}
\sum^{\infty}_{c=0}Z^{\ast}(\{1\}^a,2,\{1\}^c;T)z^c
&=\sum^{\infty}_{c=0}\rho^{-1}(Z^{\shuffle}(\{1\}^a,2,\{1\}^c;T)z^c\\
&=\sum^{\infty}_{c=0}\sum^{c}_{k=0}(-1)^{c-k}\binom{a+c-k+1}{c-k}\zeta(a+c+2-k)\rho^{-1}\left(\frac{T^k}{k!}\right)z^c\\
&=\sum^{\infty}_{k=0}\sum^{\infty}_{c=0}(-1)^{c}\binom{a+c+1}{c}\sum^{\infty}_{d=1}\frac{1}{d^{a+c+2}}z^c\rho^{-1}\left(\frac{(Tz)^k}{k!}\right)\\
 &=\left(\sum^{\infty}_{d=1}\frac{1}{d^{a+2}}\sum^{\infty}_{c=0}\binom{a+c+1}{c}\left(-\frac{z}{d}\right)^c\right)\rho^{-1}(e^{Tz})\\
&=\frac{(-1)^a}{(a+1)!}\psi^{(a+1)}(z+1)A(z)^{-1}e^{Tz}.
\end{align*}
Notice that, in the last equality, 
we have used the identities 
\[
\sum^{\infty}_{d=1}\frac{1}{d^{a+2}}\sum^{\infty}_{c=0}\binom{a+c+1}{c}\left(-\frac{z}{d}\right)^c
=\sum^{\infty}_{d=0}\frac{1}{(d+z+1)^{a+2}}
=\frac{(-1)^a}{(a+1)!}\psi^{(a+1)}(z+1).
\]
This shows that 
\begin{equation}
\label{for:zetaast generating function 1}
 \sum^{\infty}_{n=0}\zeta^{\ast}(\{1\}^a,2,\{1\}^c)z^c
=\frac{(-1)^a}{(a+1)!}\psi^{(a+1)}(z+1)A(z)^{-1}.
\end{equation}
Finally, using the expansion
\[
\psi^{(a+1)}(z+1)
=\sum^{\infty}_{l=0}(-1)^{a+l}\frac{(a+l+1)!}{l!}\zeta(a+l+2)z^{l} 
\]
and 
\begin{align*}
A(z)^{-1}
&=\prod^{\infty}_{b=2}\sum^{\infty}_{k_l=0}\frac{1}{k_l!}\left(\frac{(-1)^{l+1}}{l}\zeta(l)z^l\right)^{k_l}\\
&=\sum_{k_2,k_3,\ldots\ge 0}\frac{(-1)^{3k_2+4k_3+\cdots}}{k_2!k_3!\cdots}
\left(\frac{\zeta(2)}{2}\right)^{k_2}
\left(\frac{\zeta(3)}{3}\right)^{k_3}\cdots
z^{2k_2+3k_3+\cdots}\\
&=\sum^{\infty}_{s=0}(-1)^sC_sz^s,
\end{align*}
we see that the right-hand side of \eqref{for:zetaast generating function 1} equals 
\begin{align*}
&\frac{(-1)^a}{(a+1)!}
\sum^{\infty}_{l=0}(-1)^{a+l}\frac{(a+l+1)!}{l!}\zeta(a+l+2)z^{l} 
\sum^{\infty}_{s=0}(-1)^sC_sz^s\\
&=\sum^{\infty}_{c=0}
\left\{(-1)^c\sum^{c}_{s=0}\binom{a+c-s+1}{c-s}
\zeta(a+c-s+2)C_s\right\}z^c.
\end{align*}
Comparing the coefficient of $z^c$,
we obtain \eqref{for:regzeta stuffle}.
\end{proof}

From \eqref{for:zetaast generating function 1},
one immediately obtains the following result.

\begin{theorem}
\label{thm:zast}
We have
\begin{align}
\label{for:zetaast generating function 2}
 \sum^{\infty}_{a=0}\sum^{\infty}_{c=0}
\zeta^{\ast}(\{1\}^a,2,\{1\}^c)x^{a+1}z^c
&=-\frac{\psi(z-x+1)-\psi(z+1)}{\Gamma(z+1)e^{\gamma z}}.
\end{align}
\end{theorem}


\section*{Acknowledgement}
The authors express their sincere gratitude to Professor Maki Nakasuji for fruitful discussions.
This work was supported by Grant-in-Aid for Scientific Research (C) (Grant Number: JP21K03206) and Grant-in-Aid for Early-Career Scientists (Grant Number: JP22K13900).


\end{document}